\pdfoutput=1
\documentclass{IEEEojcsys}

\makeatletter
\let\IEEEojcsys@titlefont\titlefont
\let\titlefont\@undefined
\AtBeginDocument{\let\titlefont\IEEEojcsys@titlefont}
\makeatother

\usepackage{color,array}
\usepackage{graphicx}
\usepackage{amsmath,amssymb,amsthm}
\usepackage{mathrsfs}       % \mathscr
\usepackage{mathtools}
\usepackage{arydshln}       % dashed lines in arrays
\usepackage{bbm}            % \mathbbm{1}
\usepackage{tikz}
\usetikzlibrary{positioning,calc,arrows,shapes,shadows}
\usepackage{pgfplots}
\pgfplotsset{compat=1.17}
\usepackage{import}
\usepackage{nicefrac}
\usepackage{etoolbox}       % used below to patch the class page styles
\usepackage[colorlinks,urlcolor=blue,linkcolor=blue,citecolor=blue]{hyperref}
\usepackage{MnSymbol}

\jvol{}
\jnum{}
\paper{}
\pubyear{2026}
\receiveddate{}
\accepteddate{}
\publisheddate{}
\currentdate{}
\doiinfo{}

\makeatletter
\def\@revised{}
\let\IEEEojcsys@receivedfont\receivedfont
\def\receivedfont{\IEEEojcsys@receivedfont Preprint}
\makeatother

\makeatletter
\patchcmd{\ps@headings}{VOLUME\ \@jvol\ \@pubyear}{PREPRINT\ \@pubyear}{}{%
  \ClassWarningNoLine{IEEEojcsys}{Could not patch odd/even running footer}}
\patchcmd{\ps@headings}{VOLUME\ \@jvol\ \@pubyear}{PREPRINT\ \@pubyear}{}{%
  \ClassWarningNoLine{IEEEojcsys}{Could not patch odd/even running footer}}
\patchcmd{\ps@plain}{VOLUME\ \@jvol,\ \@pubyear}{PREPRINT\ \@pubyear}{}{%
  \ClassWarningNoLine{IEEEojcsys}{Could not patch title-page footer}}
\patchcmd{\ps@plain}{VOLUME\ \@jvol,\ \@pubyear}{PREPRINT\ \@pubyear}{}{%
  \ClassWarningNoLine{IEEEojcsys}{Could not patch title-page footer}}
\ps@headings
\makeatother

\newtheorem{theorem}{Theorem}
\newtheorem{lemma}[theorem]{Lemma}
\newtheorem{proposition}[theorem]{Proposition}

\newtheorem{definition}[theorem]{Definition}
\newtheorem{assumption}[theorem]{Assumption}
\newtheorem{remark}[theorem]{Remark}
\newtheorem{example}[theorem]{Example}

\def\dn{1ex}
\tikzset{
  auto,
  block/.style={rectangle, draw, drop shadow, fill=white,
                 minimum height=#1, minimum width=#1, inner sep=\dn},
  block/.default=5ex,
  >={latex},
  every path/.style={rounded corners},
}

\newcommand{\bbC}{\mathbb{C}}

\newcommand{\bbE}{\mathbb{E}}

\newcommand{\bbN}{\mathbb{N}}

\newcommand{\bbR}{\mathbb{R}}

\newcommand{\bH}{\boldsymbol{H}}

\newcommand{\bL}{\boldsymbol{L}}
\newcommand{\bM}{\boldsymbol{M}}
\newcommand{\bN}{\boldsymbol{N}}

\newcommand{\bPi}{\boldsymbol{\Pi}}

\newcommand{\bg}{\boldsymbol{g}}
\newcommand{\bh}{\boldsymbol{h}}

\newcommand{\bw}{\boldsymbol{w}}

\newcommand{\bz}{\boldsymbol{z}}
\newcommand{\bone}{\boldsymbol{1}}

\newcommand{\calA}{\mathcal{A}}
\newcommand{\calB}{\mathcal{B}}
\newcommand{\calC}{\mathcal{C}}
\newcommand{\calD}{\mathcal{D}}

\newcommand{\calF}{\mathcal{F}}

\newcommand{\calK}{\mathcal{K}}

\newcommand{\calP}{\mathcal{P}}

\newcommand{\calT}{\mathcal{T}}

\newcommand{\scrI}{\mathscr{I}}

\newcommand{\scrS}{\mathscr{S}}

\newcommand{\ovl}{\overline}
\newcommand{\subsEq}{\subseteq}
\DeclareMathOperator{\diag}{diag}

   \everymath=\expandafter{\the\everymath\displaystyle}
   \IfFileExists{scrextend.sty}{
     \usepackage[fontsize=10.000000pt]{scrextend}
   }{
     \renewcommand{\normalsize}{\fontsize{10.000000}{12.000000}\selectfont}
     \normalsize
   }
   \usepackage{amsmath}\usepackage{amssymb}
   \makeatletter\@ifpackageloaded{underscore}{}{\usepackage[strings]{underscore}}\makeatother

\begin{document}

\sptitle{Regular Paper}

\title{Stochastic Gradient Descent with Momentum:
       Analysis and Synthesis via Integral Quadratic Constraints}

\author{DENNIS GRAMLICH\affilmark{1},
        CARSTEN W. SCHERER\affilmark{2},
        CHRISTIAN EBENBAUER\affilmark{1}}

\affil{Chair of Intelligent Control Systems, RWTH Aachen University,
       52074 Aachen, Germany}
\affil{Chair of Mathematical Systems Theory, University of Stuttgart,
       70569 Stuttgart, Germany}

\corresp{Corresponding author: Dennis Gramlich
         (\href{mailto:dennis.gramlich@ic.rwth-aachen.de}%
               {dennis.gramlich@ic.rwth-aachen.de})}

\authornote{This work was supported by the German Research Foundation DFG
            under Grant EB425/9-1. C. Scherer is funded by Deutsche Forschungsgemeinschaft (DFG, German Research Foundation) under Germany's Excellence Strategy - EXC 2075 – 390740016. He acknowledges the support by the Stuttgart Center for Simulation Science (SimTech).}

\markboth{GRAMLICH ET AL.}{SGD WITH MOMENTUM AND STOCHASTIC IQCs}

%===============================================================================
\begin{abstract}
This article applies \emph{dynamic} integral quadratic constraints (IQCs) to
the analysis and synthesis of accelerated stochastic gradient
algorithms.
We consider composite objective functions whose gradient can be
approximated via mini-batch sampling and we model the resulting stochastic gradient oracle as a feedback nonlinearity
in the spirit of Lur'e systems from robust control.
Our first main contribution is a family of IQCs that characterize the second-order statistics of mini-batch gradients, extending the classical Zames--Falb multipliers.
Our second contribution is a semidefinite-program for
certifying exponential convergence rates of stochastic gradient algorithms, and a complementary condition for bounding the asymptotic variance caused by non-vanishing gradient noise.
Our third contribution is a convex synthesis procedure that identifies mini-batch gradient algorithms with
the smallest certifiable convergence rate.
% Numerical experiments demonstrate that the synthesized algorithms approach the
% optimal deterministic Triple Momentum rate as the batch size \(b\) grows, and match
% the performance of $b$ sequential gradient descent steps for large condition
% numbers.
\end{abstract}


\maketitle

%===============================================================================
\section{Introduction}
\label{sec:introduction}
%===============================================================================

We analyze mini-batch gradients and their use in stochastic optimization
methods such as stochastic gradient descent (SGD).
Mini-batch gradients are particularly important for
\emph{composite} objective functions of the form
\begin{align}
   z \mapsto f(z) = \frac{1}{N} \sum_{i=1}^N f_i(z). \label{eq:compositeObjective}
\end{align}
Since $N$ is typically large, evaluating the full gradient $z \mapsto \nabla f(z)$
is computationally expensive.
Instead, stochastic gradient methods rely on mini-batch gradients
\begin{align}
   z \mapsto \nabla f_{\scrI}(z) = \frac{1}{|\scrI|} \sum_{i \in \scrI} \nabla f_i(z),
   \label{eq:sampledGrad}
\end{align}
where the mini-batch $\scrI \subseteq \{1, \ldots, N\}$ is a randomly sampled index set with size
$b = |\scrI|$.
For $b \ll N$, the mini-batch gradient $\nabla f_{\scrI}$ requires
significantly fewer computations than the full gradient, but introduces
uncertainty (gradient noise).

We consider two mini-batch sampling schemes: uniform sampling with replacement, where \(\scrI\) is generated by \(b\) uniform i.i.d. random variables from \(\{1, \ldots, N\}\), and uniform sampling without replacement, where \(\scrI\) is chosen uniformly from all possible subsets of \(\{1, \ldots, N\}\) of size $b$. When sampling with replacement, $\scrI$ is technically a random multiset that may contain repeated indices.
Anticipating later insights, we note that sampling with replacement yields potentially looser convergence bounds, but has
the advantage of providing estimates that are independent from the number of objective components $N$ in \eqref{eq:compositeObjective}. Since mini-batch sampling preserves the expected value,
$\nabla f(z) = \mathbb{E}[\nabla f_{\scrI}(z)]$,
the deviation of the mini-batch gradient from the full gradient can be
regarded as a zero-mean stochastic disturbance. 

We work under the following standing assumption on \eqref{eq:compositeObjective}.

\begin{definition}[$m$-convexity and $l$-concavity, \cite{scherer2023robust}]
\label{def:mlConvexity}
A differentiable function $f:\bbR^d \to \bbR$ belongs to the class
$\scrS_{m,l}$ if
\begin{align}
   f^m: z \mapsto f(z) - \tfrac{m}{2}\|z\|^2,
   \qquad
   f^l: z \mapsto f(z) - \tfrac{l}{2}\|z\|^2
   \label{eq:mlConvexity}
\end{align}
are convex and concave, respectively, for constants $m,l \in \bbR$ with $m \leq l$. The restricted class
$\scrS_{m,l}^z := \{f \in \scrS_{m,l} \mid \nabla f(z) = 0\}$
contains all functions from $\scrS_{m,l}$ with zero gradient at
the point 
$z \in \bbR^d$.
\end{definition}

\begin{assumption}
\label{ass:lmConvexity}
The functions $(f_i)_{i=1}^N$ are in $\scrS_{m,l}$ for the same constants
$0 < m \leq l$.
\end{assumption}

Since $f$ is the average of the $f_i$, it is also in $\scrS_{m,l}$.
Assuming \(f \in \scrS_{m,l}\) is the standard for analyzing linear convergence of gradient-based algorithms. It corresponds to $m$-strong convexity and $l$-Lipschitz-continuous gradients. This assumption implies that \(f\) has a global minimizer \(z_\star^f\).

\begin{definition}[Irreducible variance]
\label{def:irreducibleVariance}
The \emph{irreducible variance} is the real, non-negative constant
\begin{align*}
   \sigma_\infty^2
   \coloneq
   \frac{1}{N} \sum_{i = 1}^N \|\nabla f_i(z_\star^f)\|^2
   .
\end{align*}
\end{definition}

\begin{assumption}
\label{ass:irreducibleVariance}
The irreducible variance satisfies $\sigma_\infty^2 = 0$.
\end{assumption}

Assumption~\ref{ass:irreducibleVariance} holds in the so-called
\emph{interpolation regime}~\cite{ma2018power}, i.e., in the case where deep learning models have sufficiently many parameters to interpolate their training data.
The analysis of this article is split into two parts corresponding to two regimes: \emph{exponential convergence analysis} under Assumption~\ref{ass:irreducibleVariance}, and \emph{asymptotic variance analysis} without it.

Figure~\ref{fig:nesterovConvergence} gives an impression of these regimes for Nesterov's
Accelerated Gradient Descent (NAGD)~\cite{nesterov1983method} applied to the
objectives 
\begin{align}
   f_1(z) &= \tfrac{1}{2}\|z - 0.12\|^2,
   &
   f_2(z) = f_3(z) &= \tfrac{1}{20}\|z + 0.6\|^2
   \label{eq:specificCompositeObjective}
\end{align}
with mini-batch gradients ($b = 1$). Details on the data generation for Figure~\ref{fig:nesterovConvergence} are given in Appendix~\ref{app:detailsLowerBound}. The comparison of the convergence of NAGD for either mini-batch or full gradients reveals that the stochastic noise in the mini-batch gradients has profound effects. With mini-batch gradients, NAGD exhibits first a phase of exponential convergence (corresponding to the sloped dashed line) and, subsequently, reaches an error floor (corresponding to the flat dashed line). The error floor is what is being analyzed by the asymptotic variance analysis. We further regard the exponential convergence analysis under Assumption~\ref{ass:irreducibleVariance} as an approximate analysis of the first phase of observed exponential convergence: In this phase, while not being identical to zero, the gradient noise is small compared to \(z - z_\star^f\). Comparing the expected error of this phase with the full gradient trajectories indicates that even the initial phase of convergence of NAGD with mini-batch gradients is slower than with full gradients.

Figure~\ref{fig:asymptoticVariance} shows the asymptotic variance (error
floor) for Heavy Ball (HB)~\cite{polyak1964some}, NAGD, Gradient Descent
(GD), and Triple Momentum (TM) \cite{van2017fastest} with batch size $b = 1$ and parameters
optimized for the deterministic case.
As the condition number $\kappa = l/m$ increases, the variance grows
significantly.
Curves escaping to infinity indicate instability.
Critically, the curves shown are lower bounds on the worst-case error floor, because they are obtained for the specific instance \eqref{eq:specificCompositeObjective} of \eqref{eq:compositeObjective}.
The instability of NAGD in the stochastic setting has also been reported
in~\cite{assran2020convergence} and motivates the algorithm design
of this article.

To understand how stochastic uncertainty induces these qualitative changes,
we decompose the mini-batch gradient into the sum of
\begin{align}
   \nabla f(z), \quad
   \nabla f_{\scrI}(z) - \nabla f(z) - \nabla f_{\scrI}(z_\star^f),
   \quad
   \nabla f_{\scrI}(z_\star^f).
   \label{eq:batchGradComponents}
\end{align}
This sum recovers $\nabla f_{\scrI}(z)$.
The first component is the full gradient, driving optimization towards
$z_\star^f$.
The second is $z$-dependent random noise that vanishes at $z_\star^f$. It
is the dominating noise component far from \(z_\star^f\), because its dependence on \(z\) allows its variance to grow with the distance of \(z\) from \(z_\star^f\). This component causes the degraded convergence rate in the initial
phase in Fig.~\ref{fig:nesterovConvergence} and the instability in Fig.~\ref{fig:asymptoticVariance}.
The third component is irreducible noise present even at the optimum and responsible for the
error floor.

The described scenario is amenable to analysis via integral quadratic
constraints (IQCs)~\cite{megretski1997system}.
The use of IQCs for accelerated gradient methods is motivated by the fact that
IQCs led to the fastest deterministic gradient method for $\scrS_{m,l}$
according to~\cite{van2017fastest}.
Our contributions are as follows.
\begin{enumerate}
   \item A generalization of the Zames--Falb IQCs to mini-batch gradients
         of the form~\eqref{eq:sampledGrad}.
   \item Convex algorithm analysis and synthesis conditions based on
         these novel IQCs.
   \item Synthesized methods that are optimal in the interpolation regime in the sense that:
         
   ~\(\bullet\) As the batch size $b$ goes to infinity, the performance of the synthesized method approaches that of the optimal deterministic TM method.

   ~\(\bullet\) As the condition number $\kappa = l/m$ goes to infinity, the synthesized method approaches the convergence rate of $b$ sequential GD steps.
\end{enumerate}

\begin{figure}[t]
   \centering
   \input{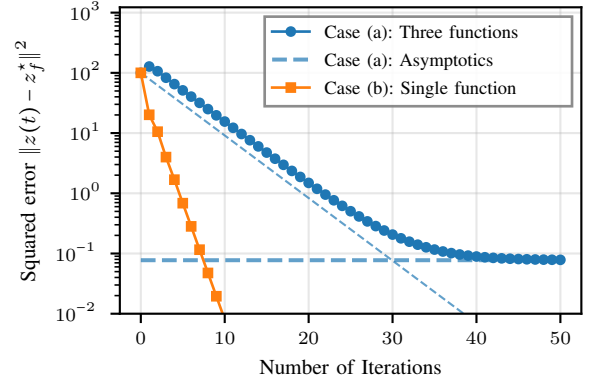}
   \caption{Expected optimization error
            $\bbE\|z_t - z_\star^f\|^2$ for NAGD with full gradients
            vs.\ mini-batch gradients ($b = 1$).
            Dashed lines show asymptotic rates: linear convergence far from
            the optimum, and an error floor near it.}
   \label{fig:nesterovConvergence}
\end{figure}

\begin{figure}[t]
   \centering
   \vspace{-4mm}
   \input{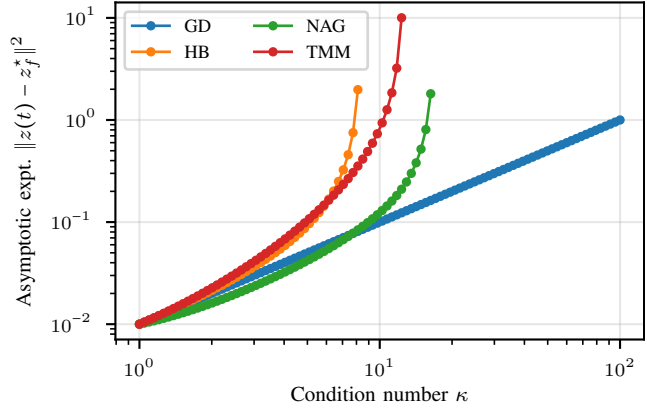}
   \caption{Asymptotic variance of stochastic gradient methods (batch size
            $b = 1$) vs.\ condition number $\kappa = l/m$.
            Parameters are optimized for the deterministic case.
            Curves escaping to infinity indicate instability.}
   \label{fig:asymptoticVariance}
\end{figure}

%===============================================================================
\section{Notation}
\label{sec:notation}
%===============================================================================

Here, we collect the principal notations used in the article.

\paragraph*{Sets and numbers}
The symbol $\bbR$ denotes the real numbers, $\bbR_{\geq 0}$ the non-negative reals, and
$\bbN_0 = \{0,1,2,\ldots\}$ the non-negative integers.
We write $\{1,\ldots,N\}$ for integer ranges.

\paragraph*{Vectors and matrices}
Column vectors are denoted by lowercase letters; matrices by uppercase.
The symbol $I_n \in \bbR^{n \times n}$ is the identity matrix and
$0$ is the zero matrix of appropriate dimension (subscripts
are omitted when clear from context).
The all-ones vector is denoted $\bone \in \bbR^n$.
For a real, symmetric matrix $A$, $\lambda_{\min}(A)$ denotes its smallest eigenvalue.
The notation $A \succ 0$ ($A \succeq 0$) means $A$ is symmetric positive (semi)definite.
A diagonal matrix with vector $\nu$ on the diagonal is denoted $\diag(\nu)$.
The Kronecker product is denoted by $\otimes$.
The notation
\begin{align*}
   \begin{pmatrix}
      x_{t+1} \\
      \hline
      y_t
   \end{pmatrix}
   =
   \left(
   \begin{array}{c|c}
      A & B \\ \hline
      C & D
   \end{array}
   \right)
   \begin{pmatrix}
      x_t \\ 
      \hline
      u_t
   \end{pmatrix}
\end{align*}
is used throughout to display a linear time-invariant (LTI) system with state
$x_t$, input $u_t$ and output $y_t$; the horizontal and vertical lines
separate the state-update from the output equation and the state from the
input, respectively.

\paragraph*{Notation for stochastics}
The symbol $\bbE[\cdot]$ denotes the expectation and $\bbE[\cdot|\mathcal{F}]$ the conditional expectation w.r.t. a given $\sigma$-algebra $\mathcal{F}$. A sequence \((y_t)_{t \in \calT}\) of random variables with \(\calT = \{0,1,\ldots,\tau\} \subsEq \bbN_0\) or \(\calT = \bbN_0\) is referred to as a stochastic process. A \emph{filtration} is a sequence of $\sigma$-algebras $(\mathcal{F}_t)_{t \in \calT}$ such that $\mathcal{F}_t \subseteq \mathcal{F}_{t+1}$ for all $t$. A stochastic process $(y_t)_{t \in \calT}$ is \emph{adapted} to a filtration $(\mathcal{F}_t)_{t \in \calT}$ if $y_t$ is $\mathcal{F}_t$-measurable for all $t \in \calT$.

\paragraph*{Signal sequences}
Bold lowercase letters denote finite sequences, i.e., $\bz = (z_0, z_1, \ldots, z_\tau)$ and similarly for $\bg$, $\bw^1$,
$\bw^2$. These symbols are identified with the column vectors obtained by stacking the sequence members, e.g., $\bz = (z_0^\top, z_1^\top, \ldots, z_\tau^\top)^\top$.
Bold uppercase letters ($\bH$, $\bL$, etc.) denote block matrices with \(d \times d\)-blocks that operate on such stacked sequences.

%===============================================================================
\section{Related Work}
\label{sec:relatedWork}
%===============================================================================

The numerical analysis and synthesis of gradient optimization algorithms via
semidefinite programming (SDP) has gained considerable attention since the
foundational work of~\cite{drori2014performance} and~\cite{lessard2016analysis}.
In~\cite{lessard2016analysis}, the feedback representation of optimization
algorithms as Lur'e systems is identified, and the IQC framework is used to
derive convergence conditions.
In~\cite{van2017fastest}, the same framework leads to the Triple Momentum
algorithm, which achieves the fastest convergence rate among gradient methods
on $\scrS_{m,l}$.
A tutorial treatment of algorithm synthesis via dynamic IQCs and SDP is
given by~\cite{scherer2025tutorial}.

The vast majority of this line of research, however, addresses the
deterministic case. Attempts to generalize to the stochastic setting include~\cite{hu2017analysis} and~\cite{taylor2019stochastic}, but these are limited to analysis. Moreover,~\cite{taylor2019stochastic} imposes restrictive global bounds on the variance of the mini-batch gradient and \cite{hu2017analysis} relies on restrictive versions of stochastic IQCs, which hold path-wise under Assumption~\ref{ass:irreducibleVariance}. The present article extends both analysis and synthesis to the stochastic
setting.

Outside of the stochastic optimization context, stochastic generalizations of IQCs are proposed in \cite{ugrinovskii1999absolute,ugrinovskii1999finite,ugrinovskii2000output,petersen2006nonlinear,aberkane2007multiobjective}. However, a consensus on the appropriate definition of stochastic IQCs is still lacking, and the definition provided in the cited works is not suitable for the synthesis of optimization algorithms of the present article: the IQCs proposed in \cite{ugrinovskii1999absolute,ugrinovskii1999finite,ugrinovskii2000output,petersen2006nonlinear,aberkane2007multiobjective} involve the closed-loop system behavior, which is not available for synthesis. In contrast, the stochastic IQCs of the present article are defined in terms of the open-loop behavior of the system component to be described with IQCs, which is available for synthesis.

The instability of accelerated methods with stochastic gradients has been
studied in~\cite{cohen2018acceleration}, which motivates the present
algorithm design.
For the non-interpolation regime ($\sigma_\infty^2 > 0$),
\cite{moulines2011non} establishes that Polyak--Ruppert averaging yields
optimal asymptotic rates, while~\cite{ghadimi2013optimal} shows that momentum methods
achieve faster initial convergence.

% Lower bounds on convergence rates for SGD are studied
% in~\cite{kidambi2018insufficiency}, which shows that momentum does not improve
% the convergence rate of SGD with batch size $b = 1$, motivating our batch GD
% reference rates in Section~\ref{sec:numerics}.

%===============================================================================
\section{Algorithm Modeling}
\label{sec:algorithmModeling}
%===============================================================================

This section develops the system-theoretic framework for gradient optimization algorithms. Starting from elementary examples, we derive the structural constraints imposed by a minimal convergence requirement. We further argue that convergence analysis can be restricted to functions with minimizer at the origin, and identify the resulting setup as the classical synthesis configuration of robust control.

\subsection{Gradient Algorithms as Lur'e Systems}
\label{sec:lureRepresentation}

It is well-known that the simple GD iterative scheme
\begin{align}
   z_{t+1} = z_t - \alpha \nabla f_{\scrI_t}(z_t), \qquad t \in \bbN_0,
   \label{eq:gradientDescent}
\end{align}
for minimizing \eqref{eq:compositeObjective} with \(\alpha \in ]0, \frac{1}{l+m}]\) has the property that the expected quadratic error \(\bbE[\|z_t - z_\star^f\|^2]\) converges to a bounded interval \cite[Corollary 3]{hu2017analysis}. We use the iteration index $t$ to emphasize that~\eqref{eq:gradientDescent}
is a discrete-time dynamical system. Here, and for any following algorithm, mini-batches \(\scrI_t\) and \(\scrI_{t'}\) with $t \neq t'$ are sampled independently.
With state $x_t := z_t$ and input $w_t := \nabla f_{\scrI_t}(z_t)$,
GD takes the form
\begin{align*}
   x_{t+1} = x_t - \alpha w_t, \quad z_t = x_t,
   \qquad w_t = \nabla f_{\scrI_t}(z_t),
\end{align*}
i.e., the feedback interconnection of a linear system
$x_{t+1} = x_t - \alpha w_t$, $z_t = x_t$
with the random nonlinearity $w_t = \nabla f_{\scrI_t}(z_t)$.
Deterministic instances of such a feedback configurations are called \emph{Lur'e systems} in control theory.

\emph{Accelerated} methods additionally employ a \emph{momentum} term.
Nesterov's algorithm or TM, for instance, read
\begin{align}
   v_{t+2} = v_{t+1} + \beta(v_{t+1} - v_t)
   - \alpha \nabla f_{\scrI_t}\bigl(v_{t+1} + \gamma(v_{t+1} - v_t)\bigr)
   \label{eq:nesterov}
\end{align}
for suitable parameters $\alpha, \beta, \gamma$~\cite{nesterov2013introductory}.
Setting $z_t := v_{t+1} + \gamma(v_{t+1}-v_t)$, $w_t := \nabla f_{\scrI_t}(z_t)$, and
$x_t := (v_{t+1}, v_t)^\top \in \bbR^{2d}$, this rewrites as the second-order LTI
system
\begin{align*}
   x_{t+1} &= \begin{pmatrix} (1+\beta)I_d & -\beta I_d \\ I_d & 0 \end{pmatrix} x_t
              + \begin{pmatrix} -\alpha I_d \\ 0 \end{pmatrix} w_t,
   \\
   z_t &= \begin{pmatrix} (1+\gamma)I_d & -\gamma I_d \end{pmatrix} x_t
\end{align*}
in feedback with $w_t = \nabla f_{\scrI_t}(z_t)$, which is again a Lur'e system.

These examples motivate the general class of mini-batch gradient algorithms
\begin{align}
   \begin{pmatrix}
     x_{t+1}\\
     \hline
     z_t
   \end{pmatrix}
   =
   \left(
   \begin{array}{c|c}
     A & B \\
     \hline
     C & 0
   \end{array}
   \right)
   \begin{pmatrix}
     x_t\\
     \hline
     w_t
   \end{pmatrix},
   \qquad
   w_t = \nabla f_{\scrI_t}(z_t),
   \label{eq:linearGradientAlgorithm}
\end{align}
where $x_t \in \bbR^{n}$ is the algorithm state,
$z_t \in \bbR^d$ is the \emph{query point} submitted to the gradient oracle,
$w_t \in \bbR^d$ is the received (mini-batch) gradient, and
$A \in \bbR^{n \times n}$, $B \in \bbR^{n \times d}$,
$C \in \bbR^{d \times n}$ are the system matrices.
A block diagram of the Lur'e system~\eqref{eq:linearGradientAlgorithm} is
shown in Fig.~\ref{fig:lureDiagram}.

\subsection{Reduction to \texorpdfstring{\(d = 1\)}{d = 1}}
We reduce to dimension $d = 1$: the case $d > 1$ follows by replacing $A$, $B$, $C$ with $A \otimes I_d$, $B \otimes I_d$, $C \otimes I_d$, because the function classes $\scrS_{m,l}$ and all subsequent analysis
conditions are compatible with this Kronecker structure (cf. \cite{lessard2016analysis,scherer2023optimization}).

\begin{figure}[t]
\centering
\tikzstyle{block} = [draw,rectangle,thick,minimum height=2em,minimum width=2em,
shade, shading=axis, top color=lightgray, bottom color=white,
rounded corners=.15cm]
\begin{tikzpicture}[scale=0.95]
  \def\bw{3.4cm}
  \def\bh{1.6cm}
  \def\gh{1.0cm}
  \def\sep{2.0cm}

  % Algorithm (LTI) block
  \node[block, drop shadow, minimum height=4em] (G) at (0,0)
     {$
     \begin{pmatrix}
      x_{t+1}\\ z_t
     \end{pmatrix}
     =
     \left(
     \begin{array}{c|c}
      A & B\\
      \hline
      C & 0
     \end{array}
     \right)
     \begin{pmatrix}
      x_t\\ w_t
     \end{pmatrix}
     $};

  % Gradient oracle block
  \node[block, drop shadow, minimum height=3em] (nabla) at (0, 2cm)
     {$\nabla f_{\scrI_t}(\cdot)$};
        \draw[->] (G.east) -- ([xshift=0.6cm]G.east) |- node[right]{$z_t$} (nabla.east);
        \draw[->] (nabla.west) -| node[left]{$w_t$} ([xshift=-0.6cm]G.west) -- (G.west);
\end{tikzpicture}
\caption{Lur'e system representation of a gradient algorithm.
         The LTI block encodes the algorithm dynamics; the nonlinear block
         evaluates the (mini-batch) gradient oracle.}
\label{fig:lureDiagram}
\end{figure}
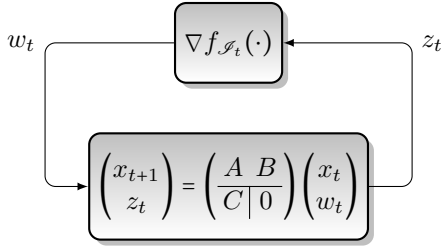

\subsection{Minimal Convergence Requirement}
\label{sec:minimalConvergence}

We derive a necessary structural condition on $(A,B,C)$ from a minimal
performance requirement on the asymptotic expected quadratic error.
Throughout, we impose the assumption that $(A,C)$ is
\emph{detectable}, i.e., that
$(A-\lambda I_n)x = 0$ and $Cx = 0$ jointly imply $x = 0$ for all $\lambda$ with $|\lambda| \geq 1$.
This is without loss of generality, because any state component that is unobservable through the query point $z_t$ does not influence the output; such components can be removed by a
state-space reduction without altering the algorithm's behavior.

We test algorithms of the form \eqref{eq:linearGradientAlgorithm} against the case $N = 1$ with the quadratic component function $f_1(z) = \frac{m}{2}(z - z_\star^f)^2$; so
$w_t = m(z_t - z_\star^f) = m(Cx_t - z_\star^f)$.
This is deterministic (no gradient noise).
The algorithm~\eqref{eq:linearGradientAlgorithm} reduces to the affine system
\begin{align}
   x_{t+1} = (A + BmC)\,x_t - Bm\,z_\star^f,
   \qquad z_t = Cx_t,
   \label{eq:quadraticAlgorithm}
\end{align}
whose quadratic error $\|z_t - z_\star^f\|^2$ is deterministic.
A minimal requirement for a sensible algorithm is that this error is bounded
for all initial conditions $x_0 \in \bbR^n$ and all $z_\star^f \in \bbR$.

\begin{lemma}[cf. \cite{michalowsky2021robust,scherer2023optimization}, Theorem~1 of \cite{scherer2025tutorial}]
   \label{lem:FixedPointCondition}
   Consider \eqref{eq:quadraticAlgorithm} with $(A,C)$ detectable. Then one of the following holds:
   \begin{enumerate}
      \item For any \(z_\star^f \in \bbR\), there exists a unique \(x_\star \in \bbR^n\) satisfying
      \begin{align}
         (A - I_n)x_\star = 0
         \qquad \text{and} \qquad
         Cx_\star = z_\star^f.
         \label{eq:equilibriumEquations}
      \end{align}
      \item For any \(c > 0\), there exist \(z_\star^f \in \bbR\) and \(x_0 \in \bbR^n\) such that
      \begin{align}
         \limsup_{t \to \infty}\|z_t - z_\star^f\|^2 \geq c.
         \label{eq:unboundedError}
      \end{align}
   \end{enumerate}
\end{lemma}
\begin{proof}
   We show that if 1) does not hold, then 2) holds. In particular, the two bullets are not exclusive.

   Denote with \(x_t(x_0,z_\star^f)\) the trajectory of~\eqref{eq:quadraticAlgorithm} starting from \(x_0\) and with parameter \(z_\star^f\). Since \eqref{eq:quadraticAlgorithm} is linear, \(x_t(x_0,z_\star^f)\) is a linear function of \(x_0\) and \(z_\star^f\). In particular, for \(x_0 = \varepsilon x_s\) and \(z_\star^f = \varepsilon z_s^f \in \bbR\), the trajectory is homogeneous in \(\varepsilon\), i.e., \(x_t(\varepsilon x_s, \varepsilon z_s^f) = \varepsilon x_t(x_s, z_s^f)\) for all \(t\). Consequently, if there exist \(x_s \in \bbR^n\) and \(z_s^f \in \bbR\) such that
   \begin{align*}
      \limsup_{t \to \infty}\|z_t(x_s, z_s^f) - z_s^f\|^2 > 0,
   \end{align*}
   then for any \(c > 0\), there exists \(\varepsilon > 0\) such that
   \begin{align*}
      c & \leq
      \varepsilon^2 \limsup_{t \to \infty}\|z_t(x_s, z_s^f) - z_s^f\|^2
      \\
      &=\limsup_{t \to \infty}\|z_t(\varepsilon x_s, \varepsilon z_s^f) - \varepsilon z_s^f\|^2
      ,
   \end{align*}
   which is the case \eqref{eq:unboundedError}.

   We conclude that if \eqref{eq:unboundedError} does not hold, then
   \begin{align}
      \limsup_{t \to \infty}\|z_t(x_0, z_\star^f) - z_\star^f\|^2 = 0, \qquad \forall z_\star^f \in \bbR, x_0 \in \bbR^n. \label{eq:stableFixedPoint}
   \end{align}
   This statement for \(z_\star^f = 0\) and  detectability of \((A,C)\) together imply that the matrix \(A + BmC\) is Schur stable. Consequently, the fixed-point equation of~\eqref{eq:quadraticAlgorithm},
   \begin{align*}
      x_\star = (A + BmC)x_\star - Bmz_\star^f,
   \end{align*}
   has a unique solution \(x_\star\) for each \(z_\star^f\). In addition, due to \eqref{eq:stableFixedPoint}, the stable fixed-point \(x_\star\) must satisfy \(C x_\star = z_\star^f\). Consequently, if \eqref{eq:unboundedError} does not hold, then \eqref{eq:equilibriumEquations} holds for all $z_\star^f$ and all $x_0$.
\end{proof}

Since \eqref{eq:unboundedError} is both undesirable and avoidable by the GD algorithm, we require that the first case of Lemma~\ref{lem:FixedPointCondition} holds for any algorithm considered in the sequel. Note that for arbitrary $z_\star^f \neq 0$, $Cx_\star = z_\star^f \neq 0$ forces $x_\star \neq 0$. So \eqref{eq:equilibriumEquations} implies $1 \in \sigma(A)$. According to \cite{scherer2023optimization}, solvability and uniqueness of~\eqref{eq:equilibriumEquations} for every $z_\star^f \in \bbR$ are equivalent to the integrator structure characterized in the next subsection.

\subsection{The Integrator Structure}
\label{sec:integratorStructure}

For detectable $(A,C)$, the condition $1 \in \sigma(A)$ derived above
forces the system matrices to have, after a suitable state-coordinate change,
the block structure
\begin{align}
   \left(
   \begin{array}{c|c}
     A & B \\
     \hline
     C & 0
   \end{array}
   \right)
   =
   \left(
   \begin{array}{cc|c}
      A_c & B_c & 0\\
      0   & 1   & 1\\
      \hline
      C_c & D_c & 0
   \end{array}
   \right),
   \quad
   \det
   \begin{pmatrix}
      A_c - I_{n_c} & B_c\\
      C_c         & D_c
   \end{pmatrix}
   \neq 0,
   \label{eq:algorithmStructure}
\end{align}
where $A_c \in \bbR^{n_c \times n_c}$, $B_c \in \bbR^{n_c \times 1}$,
$C_c \in \bbR^{1 \times n_c}$, $D_c \in \bbR$
are free parameters of the algorithm and the total state dimension is $n = n_c + 1$ (cf. \cite{scherer2023optimization} or Theorem~1 in \cite{scherer2025tutorial}).

Reading~\eqref{eq:algorithmStructure} row by row,
the system is the series interconnection of the
\emph{discrete-time integrator}
\begin{align}
   x_{t+1}^s = x_t^s + w_t, \qquad y_t = x_t^s,
   \label{eq:integratorDynamics}
\end{align}
with scalar transfer function $\frac{1}{z-1}$, and the system
\begin{align}
   \begin{pmatrix}
      x_{t+1}^c \\ z_t
   \end{pmatrix}
   =
   \begin{pmatrix}
      A_c & B_c \\ C_c & D_c
   \end{pmatrix}
   \begin{pmatrix}
      x_t^c \\ y_t
   \end{pmatrix}.
   \label{eq:controllerDynamics}
\end{align}
That \(A_c,B_c,C_c,D_c\) are to be designed suggests the interpretation of this structure as the interconnection of a plant \eqref{eq:integratorDynamics} and a controller \eqref{eq:controllerDynamics}. The combined state partitions as $x_t = \bigl((x_t^c)^\top\ x_t^s\bigr)^\top$
with controller state $x_t^c \in \bbR^{n_c}$ and
integrator state $x_t^s \in \bbR$.
The integrator accumulates the received (mini-batch) gradients $w_t$;
the controller maps the running sum $y_t = x_t^s$ to the next query point $z_t$.
The non-singularity condition in~\eqref{eq:algorithmStructure} ensures that
the pole $z=1$ of the integrator is not canceled in the series connection,
i.e., that the combined transfer function from $w_t$ to $z_t$ retains a pole
at $z = 1$.

\begin{example}[GD, HB, NAGD, TM, cf. \cite{scherer2023optimization}]
\label{ex:algorithms}
The following standard algorithms can all be written in the form
\begin{align}
   \left(
   \begin{array}{c|c}
      A & B \\
      \hline
      C & 0
   \end{array}
   \right)
   =
   \left(
   \begin{array}{cc|c}
      1+\beta & -\beta & -\alpha \\
      1       &  0     &  0 \\
      \hline
      1+\gamma & -\gamma & 0
   \end{array}
   \right), \label{eq:algorithmForm}
\end{align}
with state $x_t = (v_{t+1}, v_t)^\top$:

\begin{itemize}
   \item For \textbf{GD} pick \(\alpha = 2/(m+l)\), \(\beta = 0\), \(\gamma = 0\).

   \item For \textbf{HB} pick \(\alpha = \frac{4}{(\sqrt{l} + \sqrt{m})^2}\), \(\beta = \left(\frac{\sqrt{l} - \sqrt{m}}{\sqrt{l} + \sqrt{m}}\right)^2\), \(\gamma = 0\).

   \item For \textbf{NAGD} pick \(\alpha = \frac{4}{3l + m}\), \(\beta = \gamma = \frac{\sqrt{3l + m} - 2\sqrt{m}}{\sqrt{3l + m} + 2\sqrt{m}}\).

   \item For \textbf{TM} pick \(\alpha = \frac{\rho + 1}{l}\), \(\beta = \frac{\rho^2}{2 - \rho}\), \(\gamma = \frac{\rho^2}{(1+\rho)(2-\rho)}\), and \(\rho = 1 - \sqrt{\frac{m}{l}}\).
\end{itemize}

The state transformation $T = \begin{pmatrix} 0 & 1 \\ -1/\alpha & \beta/\alpha \end{pmatrix}$
translates~\eqref{eq:algorithmForm} into the integrator
structure~\eqref{eq:algorithmStructure} with $n_c = 1$ and parameters
\begin{align}
   A_c = \beta,
   \quad
   B_c = -\alpha,
   \quad
   C_c = \beta(1+\gamma)-\gamma,
   \quad
   D_c = -\alpha(1+\gamma).
   \label{eq:nesterovControllerParams}
\end{align}
The non-singularity condition~\eqref{eq:algorithmStructure} reads
$\det\begin{pmatrix} \beta-1 & -\alpha \\ C_c & D_c \end{pmatrix} = \alpha$,
which holds whenever $\alpha \neq 0$.
\end{example}

\subsection{Reduction to \texorpdfstring{$\scrS_{m,l}^0$}{S0(m,l)}}
\label{sec:reductionToZero}

The integrator structure does not merely guarantee a bounded asymptotic behavior for the quadratic test case: it implies that for every composite objective $(f_i)_{i=1}^N$ from $\scrS_{m,l}$, algorithm trajectories can be shifted in such a way that the critical point \(z_\star^f\) is zero.
This shift is independent of the mini-batch sequence \((\scrI_t)_{t \geq 0}\) and reduces all analysis and synthesis to functions \(f = \sum_{i=1}^N f_i\) with zero gradient at the origin.

\begin{lemma}[Reduction to $\scrS_{m,l}^0$, cf. Theorem~2 of \cite{scherer2025tutorial}]
\label{lem:reductionToZero}
For the algorithm defined by \((A,B,C)\), let an arbitrary $x_0 \in \bbR^n$ and an arbitrary composite objective $f = \frac{1}{N}\sum_{i=1}^N f_i$ with minimizer $z_\star^f$ and $(f_i)_{i=1}^N$ from $\scrS_{m,l}$ be given. Suppose $(A,B,C)$ have the integrator structure~\eqref{eq:algorithmStructure}. Define the shifted composite objective \(\tilde{f} = \frac{1}{N}\sum_{i=1}^N \tilde{f}_i\) with \(\tilde{f}_i(z) := f_i(z + z_\star^f)\). Then, there exists a unique $x_\star \in \bbR^n$ such
that, for every realization
of $(\scrI_t)_{t \geq 0}$ the following equivalence holds:
\begin{enumerate}
   \item Signals $(x_t, z_t, w_t)$  constitute a trajectory of~\eqref{eq:linearGradientAlgorithm} with oracle \(\nabla f_{\scrI_t}(z_t)\) started at $x_0$.
   \item Signals \(\tilde{x}_t := x_t - x_\star\), \(\tilde{z}_t := z_t - z_\star^f\), \(\tilde{w}_t := w_t\) constitute a trajectory of~\eqref{eq:linearGradientAlgorithm} with oracle $\nabla \tilde{f}_{\scrI_t}(\tilde{z}_t)$ started at $\tilde{x}_0 = x_0 - x_\star$.
\end{enumerate}
% In particular, the following equivalences hold for any $c > 0$ and
% $\rho \in (0,1)$:
% \begin{enumerate}
% \item\label{item:bounded}
%   $\limsup_{t\to\infty}\bbE\|z_t - z_\star^f\|^2 < \infty$ for all
%   $(f_i)\subset\scrS_{m,l}$ and all $x_0$
%   if and only if
%   $\limsup_{t\to\infty}\bbE\|z_t\|^2 < \infty$ for all
%   $(f_i)\subset\scrS_{m,l}$ so that \(\frac{1}{N}\sum_{i=1}^N f_i = f \in \scrS_{m,l}^0\) and all $x_0 \in \bbR^n$.
% \item\label{item:expConv}
%   $\bbE\|z_t - z_\star^f\|^2 \leq c\rho^{2t}\|x_0 - x_\star\|^2$ for all
%   $(f_i)\subset\scrS_{m,l}^{z_\star^f}$, all $x_0 \in \bbR^n$ and all \(z_\star^f \in \bbR\)
%   if and only if
%   $\bbE\|z_t\|^2 \leq c\rho^{2t}\|x_0\|^2$ for all
%   $(f_i)\subset\scrS_{m,l}^0$ and all $x_0 \in \bbR^n$.
% \end{enumerate}
\end{lemma}

Note that $\tilde{f} := \frac{1}{N}\sum_{i=1}^N\tilde{f}_i \in \scrS_{m,l}^0$,
while the individual $\tilde{f}_i$ need not satisfy $\nabla\tilde{f}_i(0) = 0$
unless Assumption~\ref{ass:irreducibleVariance} holds.

\begin{proof}
The integrator structure~\eqref{eq:algorithmStructure} implies that
system~\eqref{eq:equilibriumEquations} has a unique solution $x_\star \in \bbR^n$
for every $z_\star^f \in \bbR$.
Due to \eqref{eq:equilibriumEquations}, i.e., $(A - I_n)x_\star = 0$ and $Cx_\star = z_\star^f$, the equivalence
\begin{align*}
   &
   \tilde{x}_{t+1}
   = A\tilde{x}_t + Bw_t,
   \quad
   \tilde{z}_t = C\tilde{x}_t
   \\
   \Leftrightarrow &
   x_{t+1} = Ax_t + Bw_t,
   \quad
   z_t = Cx_t
\end{align*}
holds for \(\tilde{x}_t = x_t - x_\star\) and \(\tilde{z}_t = z_t - z_\star^f\) independent of the mini-batches \(\scrI_t\).

Similarly, the definition \(\tilde{f}_i(z) \coloneq f_i(z + z_\star^f)\) implies that \(w_t = \nabla f_{\scrI_t}(z_t)\) if and only if \(w_t = \nabla \tilde{f}_{\scrI_t}(\tilde{z}_t)\).
This establishes the sample-path correspondence.
\end{proof}

\subsection{Plant-Controller Decomposition}
\label{sec:plantControllerDecomp}

Henceforth we restrict attention to $f \in \scrS_{m,l}^0$ for the composite objective and \(f_i \in \scrS_{m,l}\) for the component functions.

The integrator structure reveals the \emph{control-theoretic} nature of the
algorithm design problem.
The system~\eqref{eq:linearGradientAlgorithm}
decomposes into three components:
\begin{enumerate}
   \item \textbf{Stochastic nonlinearity:}
         $w_t = \nabla f_{\scrI_t}(z_t)$.
   \item \textbf{Plant:}
         The discrete-time integrator~\eqref{eq:integratorDynamics}.
   \item \textbf{Controller:}
         The system~\eqref{eq:controllerDynamics} encodes the free design parameters
         $(A_c, B_c, C_c, D_c)$ of the algorithm.
\end{enumerate}
This mirrors the standard synthesis configuration in robust control:
an uncertain nonlinearity, a fixed plant, and a to-be-designed controller.
The plant in our setting is a \emph{single integrator}.
Crucially, the plant is the same for every algorithm in the
class~\eqref{eq:algorithmStructure}: it is determined entirely by the
optimization structure, not by the choice of algorithm.
The synthesis problem in Section~\ref{sec:synthesis} amounts to choosing
the controller parameters $(A_c, B_c, C_c, D_c)$ so as to guarantee a
desired convergence rate uniformly over the class $\scrS_{m,l}^0$.

\subsection{The Combined System and Auxiliary Filter}
\label{sec:combinedSystem}

For the convergence rate analysis, the algorithm~\eqref{eq:linearGradientAlgorithm}
must be augmented so that its outputs allow the application of the
stochastic Zames--Falb theorem (Theorem~\ref{thm:ZamesFalb}).
This requires constructing the filtered signal
\begin{align}
   \bar{g}_t^l = \sum_{j = 0}^s \rho^{j} \lambda_j g_{t-j}^l,
   \label{eq:filteredOutput}
\end{align}
where $g_{t-j}^l = 0$ for $t - j < 0$, with filter order $s \in \bbN_0$ and
convergence rate parameter $\rho \in (0,1]$. This signal $\bar{g}_t^l$ is a weighted, exponentially decaying combination of
past gradients
\begin{align}
   g_t^m &:= \nabla f(z_t) - m z_t,
   \label{eq:mGradient}\\
   g_t^l &:= \nabla f(z_t) - lz_t
   \label{eq:lGradient}
\end{align}
of the functions \(f^m\) and \(f^l\) defined in \eqref{eq:mlConvexity}.
The filtered signal $\bar{g}_t^l$ enables the inclusion of correlations between the current $g_t^m$ and past $g_{t-j}^l$ in the analysis, which is crucial for obtaining tight convergence rate bounds.

To realize~\eqref{eq:filteredOutput} as an LTI system, introduce the filter
state $x_t^\psi \in \bbR^s$ and the combined state
\begin{align}
   \chi_t = \begin{pmatrix}
      \rho^{-t} x_t^\top & (x_t^\psi)^\top
   \end{pmatrix}^\top \in \bbR^{n+s}.
   \label{eq:combinedState}
\end{align}
The factor $\rho^{-t}$ rescales the algorithm state so that whenever \(\chi_t\) is bounded, the original algorithm state decays at rate $\rho$; see~\cite{desoer2009feedback}.
The combined state satisfies the dynamics
\begin{align}
   \begin{pmatrix}
      \chi_{t+1}\\
      \hline
      -\rho^{-t}\bar{g}_t^l
   \end{pmatrix}
   =
   \left(
   \begin{array}{c|ccc}
      \calA & \calB_1 & \calB_2 & \calB_2\\
      \hline
      \calC & \calD_1 & \calD_2 & \calD_2
   \end{array}
   \right)
   \begin{pmatrix}
      \chi_t\\
      \hline
      \rho^{-t}g_t^m\\
      \rho^{-t}w_t^1\\
      \rho^{-t}w_t^2
   \end{pmatrix}
   \label{eq:filter}
\end{align}
with the combined system matrix
\begin{align}
   \left(
   \begin{array}{c|c:c}
      \calA & \calB_1 & \calB_2\\
      \hline
      \calC & \calD_1 & \calD_2
   \end{array}
   \right)
   =
   \left(
   \begin{array}{cc|c:c}
      \rho^{-1}(A + BmC) & 0   & \rho^{-1}B & \rho^{-1}B\\
      (l-m)B_\psi C          & A_\psi & -B_\psi       & 0 \\
      \hline
      (l-m)D_\psi C          & C_\psi & -D_\psi       & 0
   \end{array}
   \right),
   \label{eq:combinedAlgorithmAndFilter}
\end{align}
where the last block column in \eqref{eq:filter} is repeated and the \emph{filter matrices} $(A_\psi, B_\psi, C_\psi, D_\psi)$ realize the
finite impulse response (FIR) filter
\begin{align}
   \left(
      \begin{array}{c|c}
         A_\psi & B_\psi\\
         \hline
         C_\psi & D_\psi
      \end{array}
   \right)
   =
   \left(
   \begin{array}{cccc|c}
      0      & 1      & \cdots & 0 & 0\\
      \vdots & \ddots & \ddots & \vdots & \vdots\\
      0      & \cdots & 0      & 1 & 0\\
      0      & \cdots & 0      & 0 & 1\\
      \hline
      \lambda_s & \lambda_{s-1} & \cdots & \lambda_1 & \lambda_0
   \end{array}
   \right)
   \label{eq:filterRealization}
\end{align}
of order $s \in \bbN_0$.
Here, $A_\psi \in \bbR^{s \times s}$ is the companion shift matrix,
$B_\psi \in \bbR^{s \times 1}$, $C_\psi \in \bbR^{1 \times s}$, and $D_\psi = \lambda_0$
is the leading filter coefficient.

The underlying structure of~\eqref{eq:filter} is depicted in Figure~\ref{fig:signalInterconnection} and can be understood
as follows.
The inputs $g_t^m = \nabla f(z_t) - mz_t$, \(w_t^1 = \nabla f_{\scrI_t}(z_t) - \nabla f(z_t) - \nabla f_{\scrI_t}(z_\star^f)\), and \(w_t^2 = \nabla f_{\scrI_t}(z_\star^f)\) essentially correspond to the gradient decomposition \eqref{eq:batchGradComponents}. The difference is that the gradient \(g_t^m\) is shifted by the term $mz_t = mCx_t$. This is compensated by absorbing \(mC\) into the \((1,1)\) block of \eqref{eq:combinedAlgorithmAndFilter}, yielding $A + BmC$ as the effective state-update matrix.
The specific roles of the block-column groups are:
\begin{itemize}
   \item The $(1,1)$ block $\rho^{-1}(A + BmC)$ propagates the
         scaled algorithm state, absorbing the slope-shift $mz_t$ of the gradient into the state update.
   \item The lower block rows involving $(A_\psi, B_\psi, C_\psi, D_\psi)$ accumulate the past values of $g_t^m$ required for the filter~\eqref{eq:filteredOutput}.
   \item The column $\calB_2$ maps the stochastic noise terms $w_t^1$ or \(w_t^2\) into the state.
\end{itemize}

\begin{figure}
   \centering
   \tikzstyle{block} = [draw,rectangle,thick,minimum height=2em,minimum width=2em,
   shade, shading=axis, top color=lightgray, bottom color=white,
   rounded corners=.15cm]
   \begin{tikzpicture}[>=latex, every path/.style={rounded corners},
                          every node/.style={font=\small}]

        %--- Algorithm loop -------------------------------------------------
        \node[block] (nl) at (-1,0) %,opacity=0.3
          {$\begin{pmatrix}
              \nabla f(z_t)\\[0.1em]
              \nabla f_{\scrI_t}(z_t)\!-\!\nabla f(z_t)\!-\!\nabla f_{\scrI_t}(0)\\[0.1em]
              \nabla f_{\scrI_t}(0)
            \end{pmatrix}$};
        \node[block=3em, minimum width=6em, align=center] (P) at (2.3,-2)
          {$x^s_{t+1} = x^s_t\!+\!w_t$\\[0.15em]$y_t = x^s_t~~~$};
        \node[block] (K) at (-2.3,-2)
          {$\left[\!\begin{array}{c|c}
              A_c & B_c \\ \hline C_c & D_c
            \end{array}\!\right]$};
        \node[circle, draw, fill=white, inner sep=2pt]
          (sum) at (3.3,0) {$+$};

        %--- Analysis: 2x2 transform ----------------------------------------
        \node[block] (transform) at (1.4, 2.05)
          {$\!\begin{pmatrix}
              I\!   & {-}mI  \\
              {-}I\!\! & lI
            \end{pmatrix}\!$};

        %--- Analysis: filter block (between transform and rho blocks) -------
        \node[block] (filter) at (-1.0, 1.7)
          {$\!\left[\!\begin{array}{c|c}
              A_\psi & B_\psi \\ \hline C_\psi & D_\psi
            \end{array}\!\right]\!$};

        %--- rho^{-t} scaling blocks (left of filter) -----------------------
        \node[block=1em] (rho1) at (-2.8, 3.8) {$\rho^{-t}$};
        \node[block=1em] (rho2) at (-2.8, 3.1) {$\rho^{-t}$};
        \node[block=1em] (rho3) at (-2.8, 2.4) {$\rho^{-t}$};
        \node[block=1em] (rho4) at (-2.8, 1.7) {$\rho^{-t}$};

        %=== ALGORITHM LOOP CONNECTIONS ======================================
        \coordinate (ztj) at (-3.7, -0.0);
        \draw (K.west) -| (ztj);
        \draw[->] (ztj) -- (nl.west);
        \fill[black] (ztj) circle[radius=2pt];

        \coordinate (p_gt) at ([yshift= 0.5cm]nl.east);
        \coordinate (p_w1) at (nl.east);
        \coordinate (p_w2) at ([yshift=-0.5cm]nl.east);
        \coordinate (q_gt) at ([xshift= 1.3cm]p_gt);
        \coordinate (q_w1) at ([xshift= 1.4cm]p_w1);
        \coordinate (q_w2) at ([xshift= 1.5cm]p_w2);
        \fill[black] (q_gt) circle[radius=2pt];
        \fill[black] (q_w1) circle[radius=2pt];
        \fill[black] (q_w2) circle[radius=2pt];

        \draw[->] (p_gt) -| node[above]{$g_t$}   (sum.north);
        \draw[->] (p_w1) -- node[above]{$w_t^1$} (sum.west);
        \draw[->] (p_w2) -| node[below, pos = 0.25]{$w_t^2$} (sum.south);

        \draw[->] (sum.east) -| node[left, pos=0.8]{$w_t$}
          ([xshift=0.4cm]P.east) -- (P.east);
        \draw[->] (P.west) -- node[below]{$y_t$} (K.east);

        %=== ANALYSIS CHAIN CONNECTIONS ======================================

        %--- z_t → transform top input
        \draw[->] (ztj) |- node[above, pos=0.5]{$z_t$}  (2.5,0.95) |-
          ([yshift=-0.3cm]transform.east);

        %--- g_t → transform bottom input
        \draw[->] (q_gt) |- ([yshift=0.3cm]transform.east);

        %--- z_t^l (top output) → filter
        \draw[->] ([yshift=-0.3cm]transform.west) --
          node[above, pos=0.5]{$-g_t^l$} ([yshift=0.05cm]filter.east);

        %--- g_t^m (bottom output) → rho2 (horizontal at y=2.7, safe below transform)
        \draw[->] ([yshift=0.35cm]transform.west) --
          node[above, pos=0.5]{$g_t^m$} (rho3.east);

        %--- filter → rho1
        \draw[->] (filter.west) --
          node[above, pos=0.5]{$-\bar{g}_t^l$} (rho4.east);

        %--- w_t^1 bypass: up to y=2.0, then left to rho3
        \draw[->] (q_w1) -- ([yshift=2.0cm]q_w1) |-
          node[above, pos=0.7]{$w_t^1$} (rho2.east);

        %--- w_t^2 bypass: up to y=2.0, then left to rho4
        \draw[->] (q_w2) -- ([yshift=2.5cm]q_w2) |-
          node[above, pos=0.7]{$w_t^2$} (rho1.east);

        %--- rho output arrows (no target blocks)
        \draw[->] (rho1.west) -- node[above, pos=0.6]{\small$\rho^{-t}w_t^2$} +(-0.9,0);
        \draw[->] (rho2.west) -- node[above, pos=0.6]{\small$\rho^{-t}w_t^1$}       +(-0.9,0);
        \draw[->] (rho3.west) -- node[above, pos=0.6]{\small$\rho^{-t}g_t^m$}       +(-0.9,0);
        \draw[->] (rho4.west) -- node[above, pos=0.6]{\small$-\rho^{-t}\bar{g}_t^l$}       +(-0.9,0);

      \end{tikzpicture}
      \caption{
      This block diagram shows how the various signals arising in algorithm analysis emerge. The algorithm is the feedback interconnection between the integrator with state \(x_t^s\), the \emph{controller} denoted as the square bracket matrix with blocks \(A_c, B_c, C_c, D_c\) and the three components of the mini-batch gradient. The input \(z_t\) and first output \(g_t\) of the component-wise uncertainty are transformed statically to produce \(g_t^m\) and \(g_t^l\). The signal \(g_t^l\) is subsequently passed through the filter 
      denoted as the square bracket matrix with blocks \(A_\psi, B_\psi, C_\psi, D_\psi\) to produce \(\bar{g}_t^l\). Finally, the resulting inputs and outputs of the uncertainty are exponentially weighted with \(\rho^{-t}\) to arrive at the signals required for Theorem~\ref{thm:ZamesFalb}.
      }
      \label{fig:signalInterconnection}
\end{figure}
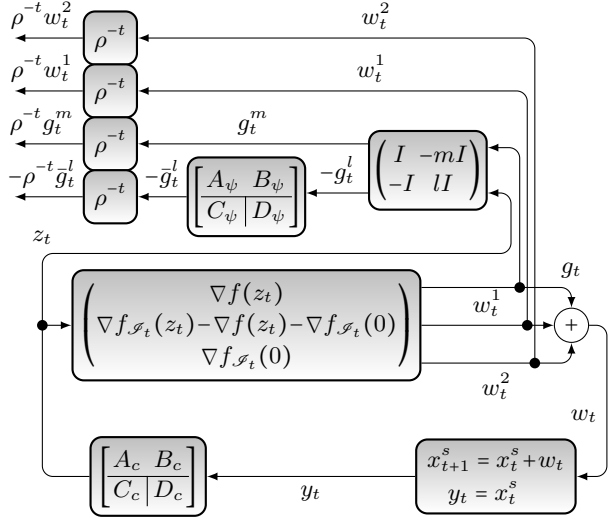

%===============================================================================
\section{Stochastic IQCs}
\label{sec:stochastiIQCs}
%===============================================================================

For the analysis and synthesis of optimization algorithms, this article relies on IQCs. In the deterministic setting, IQC families are quadratic forms that are non-negative on the graph of a troublesome system component, which is the gradient operator $\nabla f$ in the context of optimization. In the stochastic setting, the troublesome component is the mini-batch gradient operator $\nabla f_{\scrI}$ with random mini-batch $\scrI$.

IQCs enable the verification of Lyapunov functions, dissipation inequalities, and related constructs with SDP-based control theory.

The first main result is a family of stochastic IQCs capturing
the statistical properties of the components~\eqref{eq:batchGradComponents}.

\subsection{Setup and signal definitions}
\label{subsec:signalDef}

Consider a sequence of mini-batches $\scrI_0, \scrI_1, \scrI_2, \ldots$ of
uniform size $b = |\scrI_t|$, where $t$ runs from $0$ to $\tau$.
Let $(\calF_t)_{t \geq 0}$ be a filtration to which the mini-batches are
adapted, and let $\bz = (z_0, z_1, \ldots)$ be a $\calF_{t-1}$-measurable (predictable) stochastic process with bounded second moments.
The $\calF_{t-1}$-measurability of the query points is given in this articles' optimization context \cite{arjevani2023lower}: the iterate $z_t$ is computed from previous oracles before observing $\scrI_t$.

For any algorithm \eqref{eq:algorithmForm}, define the stochastic processes
$\bg = (g_0, g_1, \ldots)$ and $\bw^k = (w_0^k, w_1^k, \ldots)$ for
$k = 1,2$ via
\begin{subequations}
\label{eq:gradientSignals}
\begin{align}
   g_t   &= \nabla f(z_t),\label{eq:fullGrad}\\
   w_t^1 &= \nabla f_{\scrI_t}(z_t)
            - \nabla f(z_t)
            - \nabla f_{\scrI_t}(z_\star^f),\label{eq:noise1}\\
   w_t^2 &= \nabla f_{\scrI_t}(z_\star^f).\label{eq:noise2}
\end{align}
\end{subequations}
Note that $w_t^1$ and $w_t^2$ satisfy $\bbE[w_t^1 \mid \calF_{t-1}] = 0$ and $\bbE[w_t^2 \mid \calF_{t-1}] = 0$ (Lemma~\ref{lem:matringaleDifference}), so their products with any $\calF_{t-1}$-measurable (predictable) process have zero expectation.

\subsection{Finite-horizon IQC}

Our key result is a family of stochastic IQCs over a horizon
$t = 0, \ldots, \tau$.
Write $\bz, \bg, \bw^1, \bw^2 \in \bbR^{(\tau+1)d}$ for the stacked vectors
of the respective processes.
The IQC takes the form
\begin{align}
   &\bbE
   \left[
   \begin{pmatrix}
      \bz\\ \bg\\ \bw^1\\ \bw^2
   \end{pmatrix}^\top
   \bPi
   \begin{pmatrix}
      \bz\\ \bg\\ \bw^1\\ \bw^2
   \end{pmatrix}
   \right] \geq -\sigma_\infty^2 \bone^\top \mu, \quad \text{where}
   \label{eq:stochasticIQC}
   \\
   &
   \bPi =
   \begin{pmatrix}
      -ml(\bH^\top + \bH) & (m+l)\bH & 0 & 0\\
      (m+l) \bH^\top & -(\bH + \bH^\top) & 0 & 0\\
      0 & 0 & -2\tilde{b}\bN & 0\\
      0 & 0 & 0 & -\tilde{b}\bM
   \end{pmatrix}, \nonumber
\end{align}
where the \emph{generalized batch size} $\tilde{b}$ is
\begin{itemize}
   \item $\tilde{b} = b$, if the mini-batches $\scrI_t$ are sampled with
         replacement and
   \item $\tilde{b} = \dfrac{b(N-1)}{N-b}$, if sampled without replacement.
\end{itemize}
The irreducible variance defined in Assumption~\ref{ass:irreducibleVariance} is \(\sigma_\infty^2\).

\begin{theorem}[Finite-horizon batch gradient IQC]
\label{thm:stochasticIQCFiniteHorizon}
Let Assumption~\ref{ass:lmConvexity} hold and \(\nabla f(0) = 0\).
Let the random variables $\bz$, $\bg$, $\bw^1$, $\bw^2$ on \(\bbR^{(\tau+1)d}\) be the stacked signal vectors \eqref{eq:fullGrad}-\eqref{eq:noise2} for $t = 0, \ldots, \tau$.
Let $H \in \bbR^{(\tau+1) \times (\tau+1)}$ and
$\nu, \mu \in \bbR_{\geq 0}^{\tau+1}$ be such that the (element-wise) inequalities
\begin{align*}
   \nu \leq H \bone, \quad
   \nu \leq H^\top \bone, \quad
   0 \geq H - \diag H, \quad
   \mu \geq 0
\end{align*}
are satisfied.
Define the matrices \(\bH = H \otimes I_d\), \(\bM = \diag(\mu) \otimes I_d\), and \(\bN = \diag(\nu) \otimes I_d\). Then the quadratic constraint~\eqref{eq:stochasticIQC} holds.
\end{theorem}

The proof is given in Appendix~\ref{app:proofs}.

\subsection{Causal time-invariant IQCs}

The next result specializes Theorem~\ref{thm:stochasticIQCFiniteHorizon} to
the causal time-invariant case needed for convergence rate analysis.

\begin{theorem}[Causal time-invariant batch gradient IQC]
\label{thm:ZamesFalb}
Let Assumption~\ref{ass:lmConvexity} hold and \(\nabla f(0) = 0\).
Let $g_t^m$ and $g_t^l$ be as in~\eqref{eq:mGradient}--\eqref{eq:lGradient}.
Let $\lambda_0,\ldots,\lambda_s \in \bbR$ be a sequence with
\begin{align*}
   \lambda_t \leq 0 \text{ for } t > 0, \quad \text{and} \quad
   \sum_{t = 0}^s \rho^{-t} \lambda_t \geq 0
\end{align*}
for some given $\rho \in (0,1]$.
Then, for any \(\calF_{t-1}\)-measurable sequence $\bz = (z_t)_{t=0}^\infty$ of
$\bbR^d$-valued random variables and any finite horizon $\tau$,
\begin{align*}
   \sum_{t=0}^\tau \sum_{k=0}^s \rho^{k-2t}\lambda_k
   \bbE \big[(g_t^m)^\top g_{t-k}^l\big]
   +
   \tilde{b}\Bigl(\sum_{t = 0}^s \rho^{-t} \lambda_t\Bigr)
   \sum_{t = 0}^\tau \rho^{-2t}\bbE \|w_t^1\|^2
\end{align*}
is non-positive, where $g_{t-k}^l = 0$ for $t-k < 0$.
\end{theorem}

The proof is given in Appendix~\ref{app:proofs}.

\begin{remark}[The role of previsibility]
\label{rem:philosophyStochasticIQCs}
The assumption of \(\calF_{t-1}\)-measurable (previsible) query points $\bz$ in the IQC is essential. Indeed, for most optimization algorithms, the query point $z_t$ depends on the past mini-batches $\scrI_0, \ldots, \scrI_{t-1}$ (cf. \cite{nemirovski2009robust}), so $z_t$ is a $\calF_{t-1}$-measurable random variable. In particular, considering \(\bz\) as deterministic sequence or as independent of the mini-batches would not capture the true statistical structure of the problem. On the other hand, arbitrary random $\bz$ can be too conservative: In the latter case, the entire sequence \(\bz\) could be zero except for the worst-case realization of the entire mini-batch sequence. The IQCs would thus have to capture the worst-case behavior of the mini-batch gradient operator instead of the average behavior. The \(\calF_{t-1}\)-measurable setting strikes the right balance.
\end{remark}

%===============================================================================
\section{Convergence Rate Analysis}
\label{sec:analysis}
%===============================================================================

We present two types of convergence analysis: exponential convergence analysis
under Assumption~\ref{ass:irreducibleVariance} (interpolation regime),
and the asymptotic variance bound without Assumption~\ref{ass:irreducibleVariance}.

\subsection{Convergence rate in the interpolation regime}

The following proposition gives a sufficient condition for exponential convergence in terms of an LMI.
% The additional slack variables $Q_{13}$, $Q_{23}$ capture the stochastic independence structure.

\begin{proposition}
\label{prop:convergenceRateDynamic}
Let Assumptions~\ref{ass:lmConvexity} and~\ref{ass:irreducibleVariance} hold, let \(\nabla f(0) = 0\),
and let the combined system~\eqref{eq:combinedAlgorithmAndFilter} be as in
Section~\ref{sec:algorithmModeling}.
If there exist a matrix $P = P^\top \succ 0$ of size $(n+s) \times (n+s)$,
arbitrary matrices $Q_{13}, Q_{23} \in \bbR^{(n+s+1) \times 1}$,
and parameters $\nu, \lambda_0, \ldots, \lambda_s \in \bbR$ such that
\begin{align}
   &\begin{pmatrix}
      \calA & \calB_1 & \calB_2\\
      I     & 0       & 0
   \end{pmatrix}^\top
   \begin{pmatrix}
      P  & 0\\
      0  & -P
   \end{pmatrix}
   \begin{pmatrix}
      \calA & \calB_1 & \calB_2\\
      I     & 0       & 0
   \end{pmatrix}
   \nonumber\\
   &+
   \begin{pmatrix}
      \calC & \calD_1 & \calD_2\\
      0     & 1       & 0\\
      0     & 0       & 1
   \end{pmatrix}^\top
   \begin{pmatrix}
      0 & 1 & 0\\
      1 & 0 & 0\\
      0 & 0 & -2\tilde{b}\nu
   \end{pmatrix}
   \begin{pmatrix}
      \calC & \calD_1 & \calD_2\\
      0     & 1       & 0\\
      0     & 0       & 1
   \end{pmatrix}
   \nonumber\\
   &+
   \begin{pmatrix}
      0 & 0 & Q_{13}\\
      0 & 0 & Q_{23}\\
      Q_{13}^\top & Q_{23}^\top & 0
   \end{pmatrix}
   \prec 0,
   \label{eq:IQC_convergence}
\end{align}
$\lambda_1, \ldots, \lambda_s \leq 0$, and
$\nu = \sum_{j=0}^s \rho^{-j}\lambda_j \geq 0$ are satisfied,
then the query points of algorithm~\eqref{eq:linearGradientAlgorithm}
satisfy
\begin{align}
   \lambda_{\min}(P)\|C\|^2 \bbE \left[\|z_t\|^2\right]
   \;\leq\;
   \rho^{2t}\,
   \bbE\left[
   \begin{pmatrix} x_0 \\ 0 \end{pmatrix}^\top
   P
   \begin{pmatrix} x_0 \\ 0 \end{pmatrix}
   \right].
   \label{eq:convergence}
\end{align}
\end{proposition}

\begin{proof}
To ensure consistent handling of the initial conditions, set $x_0^\psi = 0$
and $g_t^l = 0$ for $t < 0$.

Recall that Assumption~\ref{ass:irreducibleVariance} means \(w_t^2 \equiv 0\). Pre- and post-multiply~\eqref{eq:IQC_convergence} by
$\begin{pmatrix} \chi_t^\top & \rho^{-t}(g_t^m)^\top & \rho^{-t}(w_t^1)^\top \end{pmatrix}$
and substitute the system equations~\eqref{eq:filter},
written compactly as
\begin{align*}
   \left(
   \begin{array}{c|cc}
      \calA & \calB_1 & \calB_2\\
      \hline
      \calC & \calD_1 & \calD_2
   \end{array}
   \right)
   \begin{pmatrix}
      \chi_t \\ \rho^{-t}g_t^m \\ \rho^{-t}w_t^1
   \end{pmatrix}
   =
   \begin{pmatrix}
      \chi_{t+1} \\ -\rho^{-t}\bar{g}_t^l
   \end{pmatrix}.
\end{align*}
This yields the dissipation inequality
\begin{align}
   0 \geq\; &\chi_{t+1}^\top P \chi_{t+1} - \chi_t^\top P \chi_t
   - 2\rho^{-2t}(\bar{g}_t^l)^\top g_t^m
   - 2\rho^{-2t}\tilde{b}\nu \|w_t^1\|^2 \nonumber\\
   &+ 2
   \begin{pmatrix}
      \chi_t\\ \rho^{-t}g_t^m
   \end{pmatrix}^\top
   \begin{pmatrix}
      Q_{13}\\ Q_{23}
   \end{pmatrix}
   \rho^{-t} w_t^1. \label{eq:convergenceRateDynamicDissipationInequality}
\end{align}
Taking expectations, the last term vanishes because $\chi_t$ and $g_t^m$
are $\calF_{t-1}$-measurable while $w_t^1$ satisfies $\bbE \left[ w_t^1 \middle| \calF_{t-1} \right] = 0$ (Lemma~\ref{lem:matringaleDifference}).
Summing over $t = 0, \ldots, \tau-1$ gives
\begin{align}
   \bbE[\chi_0^\top P \chi_0]
   \geq\;
   \bbE[\chi_\tau^\top P \chi_\tau]
   - 2\sum_{t=0}^{\tau-1}
     \rho^{-2t}\bbE\bigl[(\bar{g}_t^l)^\top g_t^m
     + \tilde{b}\nu \|w_t^1\|^2\bigr].
   \label{eq:expectedDissipation}
\end{align}
Substituting~\eqref{eq:filteredOutput} for $\bar{g}_t^l$ and invoking
Theorem~\ref{thm:ZamesFalb}, the sum on the right is non-negative and
can be dropped, yielding
$\bbE[\chi_\tau^\top P \chi_\tau] \leq \bbE[\chi_0^\top P \chi_0]$.
Using $z_t = Cx_t$ and $\|z_t\|^2 \leq \|C\|^2 \rho^{2t} \|\chi_t\|^2$
concludes the proof.
\end{proof}

% \begin{remark}
%    In the presented convergence analysis (Proposition~\ref{prop:convergenceRateDynamic}), the slack variables $Q_{13}, Q_{23}$ subsume the role of $\bL_{31}$, $\bL_{32}$, $\bL_{41}$, $\bL_{42}$ and $\bO$ in Remark~\ref{rem:stochasticIQC}, rendering the latter redundant. The cross-term $2(\chi_t^\top Q_{13} + \rho^{-t}(g_t^m)^\top Q_{23})\rho^{-t}w_t^1$ arising in the dissipation inequality \eqref{eq:convergenceRateDynamicDissipationInequality} vanishes in expectation for \emph{any} $Q_{13}, Q_{23}$, precisely because $\chi_t$ and $g_t^m$ are $\mathcal{F}_{t-1}$-measurable and $\bbE \left[ w_t^1 \middle| \calF_{t-1} \right] = 0$.
% \end{remark}

The slack variables $Q_{13}$, $Q_{23}$ can be eliminated using a projection lemma, yielding a smaller LMI.

\begin{theorem}[Convergence rate of batch gradient algorithms]
\label{thm:convergenceRateDynamic}
Let Assumptions~\ref{ass:lmConvexity} and~\ref{ass:irreducibleVariance} hold and \(\nabla f(0) = 0\).
If there exist $P = P^\top \succ 0$ of size $(n+s) \times (n+s)$ and
parameters $\nu, \lambda_0, \ldots, \lambda_s \in \bbR$ such that
\begin{subequations}
\label{eq:IQC_convergenceProj}
\begin{align}
   0 &\succ
   \begin{pmatrix}
      \calA & \calB_1\\
      I     & 0
   \end{pmatrix}^\top
   \begin{pmatrix}
      P  & 0\\
      0  & -P
   \end{pmatrix}
   \begin{pmatrix}
      \calA & \calB_1\\
      I     & 0
   \end{pmatrix}
   % \nonumber\\
   % &
   +
   \begin{pmatrix}
      \calC & \calD_1\\
      0     & 1
   \end{pmatrix}^\top
   \begin{pmatrix}
      0 & 1\\
      1 & 0
   \end{pmatrix}
   \begin{pmatrix}
      \calC & \calD_1\\
      0     & 1
   \end{pmatrix},
   \label{eq:IQC_convergenceProj1}\\
   0 &>
   \calB_2^\top P \calB_2 - 2\tilde{b}\nu,
   \label{eq:IQC_convergenceProj2}
\end{align}
\end{subequations}
$\lambda_1, \ldots, \lambda_s \leq 0$, and
$\nu = \sum_{j=0}^s \rho^{-j}\lambda_j \geq 0$ are satisfied,
then the query points $z_t$ of algorithm~\eqref{eq:linearGradientAlgorithm}
satisfy the bound~\eqref{eq:convergence}.
\end{theorem}

\begin{proof}
Apply the projection lemma,
Lemma~4.14 of~\cite{scherer2000linear} to~\eqref{eq:IQC_convergence}.
\end{proof}

\subsection{Irreducible noise rejection}
\label{sec:noiseRejection}

We now drop Assumption~\ref{ass:irreducibleVariance} and analyze the impact of
irreducible noise $\sigma_\infty^2 > 0$ on the asymptotic behavior.
The quantity of interest is the time-averaged expected squared distance,
$\limsup_{\tau \to \infty} \frac{1}{\tau} \sum_{t=0}^{\tau-1} \bbE\|z_t\|^2$.

\begin{proposition}
\label{prop:asymptoticVarianceDynamic}
Let Assumption~\ref{ass:lmConvexity} hold and \(\nabla f(0) = 0\). Let
$P = P^\top \succ 0$ of size $(n+s) \times (n+s)$,
arbitrary matrices $Q_{13}, Q_{14}, Q_{23}, Q_{24}$, and parameters
$\lambda_0, \ldots, \lambda_s$, $\mu$, $\nu$, \(\sigma_z^2\) exist such that
\begin{subequations}
\begin{align}
   &
   \begin{pmatrix}
      \calA & \calB_1 & \calB_2 & \calB_2\\
      I     & 0       & 0       & 0
   \end{pmatrix}^\top
   \begin{pmatrix}
      P  & 0\\
      0  & -P
   \end{pmatrix}
   \begin{pmatrix}
      \calA & \calB_1 & \calB_2 & \calB_2\\
      I     & 0       & 0       & 0
   \end{pmatrix}
   \nonumber\\
   &+
   \begin{pmatrix}
      0                       & 0 & 0 & 1\\
      \begin{psmallmatrix} C & 0 \end{psmallmatrix} & 1 & 0 & 0
   \end{pmatrix}^\top
   \begin{pmatrix}
      -\mu & 0\\
      0    & 1
   \end{pmatrix}
   \begin{pmatrix}
      0                       & 0 & 0 & 1\\
      \begin{psmallmatrix} C & 0 \end{psmallmatrix} & 1 & 0 & 0
   \end{pmatrix}
   \nonumber\\
   &+
   \begin{pmatrix}
      \calC & \calD_1 & \calD_2 & \calD_2\\
      0     & 1       & 0       & 0\\
      0     & 0       & 1       & 0
   \end{pmatrix}^\top
   \begin{pmatrix}
      0 & 1 & 0\\
      1 & 0 & 0\\
      0 & 0 & -2\tilde{b}\nu
   \end{pmatrix}
   \begin{pmatrix}
      \calC & \calD_1 & \calD_2 & \calD_2\\
      0     & 1       & 0       & 0\\
      0     & 0       & 1       & 0
   \end{pmatrix}
   \nonumber\\
   &+
   \begin{pmatrix}
      0          & 0          & Q_{13} & Q_{14}\\
      0          & 0          & Q_{23} & Q_{24}\\
      Q_{13}^\top & Q_{23}^\top & 0      & 0\\
      Q_{14}^\top & Q_{24}^\top & 0      & 0
   \end{pmatrix}
   \prec 0,
   \label{eq:dissipation}\\
   &\sigma_\infty^2 \mu \leq \tilde{b}(l-m)^2\sigma_z^2,
   \label{eq:variance}
\end{align}
\end{subequations}
and $\lambda_1, \ldots, \lambda_s \leq 0$,
$\nu = \sum_{j=0}^s \lambda_j \geq 0$ hold.
Then, the asymptotic averaged variance of $z_t$ is bounded by $\sigma_z^2$, i.e.,
\begin{align*}
   \sigma_z^2 \geq
   \limsup_{\tau \to \infty} \frac{1}{\tau}\sum_{t=0}^{\tau-1}
   \bbE\left[\|z_t\|^2\right].
\end{align*}
\end{proposition}

\begin{proof}
Set $x_0^\psi = 0$ and $g_t^l = 0$ for $t < 0$; use $\rho = 1$
in~\eqref{eq:filteredOutput}.
Pre- and post-multiply~\eqref{eq:dissipation} by
$\begin{pmatrix} \chi_t^\top & (g_t^m)^\top & (w_t^1)^\top & (w_t^2)^\top \end{pmatrix}$
and substitute the system equations to obtain the dissipation inequality
\begin{align*}
   0 \geq\; &\chi_{t+1}^\top P \chi_{t+1} - \chi_t^\top P \chi_t
   - 2(\bar{g}_t^l)^\top g_t^m
   + (g_t^m - g_t^l)^\top(g_t^m - g_t^l)\\
   &- \mu\|w_t^2\|^2 - \tilde{b}\nu \|w_t^1\|^2
   + 2
   \begin{pmatrix}
      \chi_t\\ g_t^m
   \end{pmatrix}^\top
   \begin{pmatrix}
      Q_{13} & Q_{14}\\
      Q_{23} & Q_{24}
   \end{pmatrix}
   \begin{pmatrix}
      w_t^1\\ w_t^2
   \end{pmatrix}.
\end{align*}
The last term has zero expectation (Lemma~\ref{lem:matringaleDifference}),
and $g_t^m - g_t^l = (l-m)z_t$.
Summing over $t$ and taking expectations gives
\begin{align*}
   \bbE[\chi_0^\top P \chi_0]
   &\geq\;
   (l-m)^2\sum_{t=0}^{\tau-1} \bbE\big[\|z_t\|^2\big]
   - \mu\sum_{t=0}^{\tau-1} \bbE\big[\|w_t^2\|^2\big]
   \\
   + \bbE[&\chi_\tau^\top P \chi_\tau]
   +-2 \sum_{t=0}^{\tau-1} \bbE\big[(\bar{g}_t^l)^\top g_t^m\big]
   - 2\tilde{b}\nu \sum_{t=0}^{\tau-1} \bbE \big[\|w_t^1\|^2\big]
   .
\end{align*}
Invoking Theorem~\ref{thm:ZamesFalb}, the last two sums add up to a non-negative quantity and can be dropped. Also \(\bbE[\chi_\tau^\top P \chi_\tau]\) is non-negative and can be dropped. Subsequently using the noise bound $\sum_{t=0}^{\tau-1}\bbE\|w_t^2\|^2 \leq \frac{\tau}{\tilde{b}}\sigma_\infty^2$ (from Definition~\ref{def:irreducibleVariance}) and dividing by $\tau (l-m)^2$ yields
\begin{align*}
   \frac{1}{\tau}\sum_{t=0}^{\tau-1} \bbE\|z_t\|^2
   \leq\;
   \frac{\mu \sigma_\infty^2}{(l-m)^2 \tilde{b}}+ \frac{1}{(l-m)^2 \tau}\bbE[\chi_0^\top P \chi_0]
   .
\end{align*}
Using \eqref{eq:variance} and taking $\limsup$ finishes the proof.
\end{proof}

As in the case of convergence rate analysis, the slack variables $Q_{13}, Q_{14}, Q_{23}, Q_{24}$ can be eliminated using a projection lemma, yielding a smaller LMI.

\begin{theorem}[Irreducible noise rejection]
\label{thm:asymptoticVarianceDynamic}
Let Assumption~\ref{ass:lmConvexity} hold and \(\nabla f(0) = 0\).
Assume there exist $P = P^\top \succ 0$ of size $(n+s) \times (n+s)$ and
parameters $\lambda_0, \ldots, \lambda_s$, $\mu$, $\nu$, \(\sigma_z^2\) such that
\begin{subequations}
\begin{align}
   0 \succ\;
   &
   \begin{pmatrix}
      \calA & \calB_1\\
      I     & 0
   \end{pmatrix}^\top
   \begin{pmatrix}
      P  & 0\\
      0  & -P
   \end{pmatrix}
   \begin{pmatrix}
      \calA & \calB_1\\
      I     & 0
   \end{pmatrix}
   +
   \begin{pmatrix}
      \calC & \calD_1\\
      0     & 1
   \end{pmatrix}^\top
   \begin{pmatrix}
      0 & 1\\
      1 & 0
   \end{pmatrix}
   \begin{pmatrix}
      \calC & \calD_1\\
      0     & 1
   \end{pmatrix}
   \nonumber\\
   &+
   \begin{pmatrix}
      \begin{psmallmatrix} C & 0 \end{psmallmatrix} & 1
   \end{pmatrix}^\top
   \begin{pmatrix}
      \begin{psmallmatrix} C & 0 \end{psmallmatrix} & 1
   \end{pmatrix},
   \label{eq:dissipationProj1}\\
   &
   \begin{pmatrix}
      \calB_2 & \calB_2
   \end{pmatrix}^\top
   P
   \begin{pmatrix}
      \calB_2 & \calB_2
   \end{pmatrix}
   -
   \begin{pmatrix}
      0 & 0\\
      0 & \mu
   \end{pmatrix}
   - 2\tilde{b}
   \begin{pmatrix}
      \nu & 0\\
      0 & 0
   \end{pmatrix}
   \prec 0,
   \label{eq:dissipationProj2}\\
   &\sigma_\infty^2 \mu \leq \tilde{b}(l-m)^2 \sigma_z^2,
   \label{eq:varianceProj}
\end{align}
\end{subequations}
and $\lambda_1, \ldots, \lambda_s \leq 0$,
$\nu = \sum_{j=0}^s \lambda_j \geq 0$ hold.
Then
$\sigma_z^2 \geq
\limsup_{\tau \to \infty} \frac{1}{\tau}\sum_{t=0}^{\tau-1} \bbE\left[\|z_t\|^2\right]$.
\end{theorem}

\begin{proof}
Apply the projection lemma,
Lemma~4.14 of~\cite{scherer2000linear} to~\eqref{eq:dissipation}.
\end{proof}

%===============================================================================
\section{Algorithm Synthesis}
\label{sec:synthesis}
%===============================================================================

This section derives conditions for synthesizing algorithms that minimize
the certified convergence rate $\rho$.
Specifically, we seek controller parameters $A_c, B_c, C_c, D_c$ and a
certificate matrix $P$ that render Theorem~\ref{thm:convergenceRateDynamic}
feasible for the smallest possible $\rho$.

We focus on just one of $\rho$ and $\sigma_z^2$, because optimizing $\rho$ and $\sigma_z^2$ simultaneously yields a multi-objective problem that complicates convexification.
Moreover, minimizing $\sigma_z^2$ alone is trivial (take a sufficiently small
step size at the cost of arbitrarily slow convergence), so we optimize only
$\rho$.
For notational convenience, define the functions
$q, \bar{q} : \bbR^{s+1} \to \bbR^{s+1}$ by
\begin{align}
   \begin{pmatrix}
      \lambda_0\\ \lambda_1\\ \vdots\\ \lambda_s
   \end{pmatrix}
   \mapsto
   \begin{pmatrix}
      \bar{q}(\lambda)\\ \hline q(\lambda)
   \end{pmatrix}
   =
   \begin{pmatrix}
      \lambda_s\\
      \rho \lambda_s + \lambda_{s-1}\\
      \vdots\\
      \rho^{s-1} \lambda_s + \cdots + \lambda_1\\
      \hline
      \rho^s \lambda_s + \cdots + \rho \lambda_1 + \lambda_0
   \end{pmatrix}.
   \label{eq:qFunctions}
\end{align}

\begin{theorem}[Synthesis conditions]
\label{thm:synthesisConditions}
Let $n_c = 1 + s$.
There exist controller parameters $A_c, B_c, C_c, D_c$ and a positive
definite certificate $P \in \bbR^{(n_c + 1 + s) \times (n_c + 1 + s)}$
satisfying Theorem~\ref{thm:convergenceRateDynamic} if and only if
$\rho$ and $\lambda_0, \ldots, \lambda_s$ satisfy
\begin{subequations}
\label{eq:IQC_convergenceSynth}
\begin{align}
   0 &\succ
   \begin{pmatrix}
      I & 0\\ 0 & 1\\ 0 & \rho^{-1}\\ A_\psi & -B_\psi\\ C_\psi & -D_\psi
   \end{pmatrix}^\top
   \begin{pmatrix}
      -P_{33}  & 0       & 0       & 0 & 0\\
      0        & 0       & 0       & 0 & 1\\
      0        & 0       & P_{22}  & P_{23} & 0\\
      0        & 0       & P_{32}  & P_{33} & 0\\
      0        & 1       & 0       & 0 & 0
   \end{pmatrix}
   \begin{pmatrix}
      I & 0\\ 0 & 1\\ 0 & \rho^{-1}\\ A_\psi & -B_\psi\\ C_\psi & -D_\psi
   \end{pmatrix},
   \label{eq:synthesisCondition_a}\\
   0 &<
   \rho^2 \ovl{P}_{22} - \ovl{P}_{22}
   + \frac{2ml}{(l-m)^2} q(\lambda),
   \label{eq:synthesisCondition_b}\\
   0 &\prec
   \begin{pmatrix}
      P_{22}   & P_{23}  & q(\lambda)\\
      P_{32}   & P_{33}  & -\tfrac{m}{l-m}\bar{q}(\lambda)\\
      q(\lambda) & -\tfrac{m}{l-m}\bar{q}(\lambda) & \ovl{P}_{22}
   \end{pmatrix},
   \label{eq:synthesisCondition_c}\\
   0 &\geq \lambda_1, \ldots, \lambda_s, \qquad
   0 > \rho^{-2}P_{22} - 2\tilde{b}\sum_{j=0}^s \rho^{-j}\lambda_j,
   \label{eq:synthesisCondition_d}
\end{align}
\end{subequations}
for some
$\begin{psmallmatrix} P_{22} & P_{23}\\ P_{23}^\top & P_{33} \end{psmallmatrix}
\in \bbR^{(1+s) \times (1+s)}$
and $\ovl{P}_{22} \in \bbR$. Here, \(q(\lambda)\) and \(\bar{q}(\lambda)\) are linear functions of \(\lambda\) as defined in~\eqref{eq:qFunctions} and also \(C_\psi\) and \(D_\psi\) are linear functions of \(\lambda\) as defined in~\eqref{eq:filterRealization}.
\end{theorem}

The proof is given in Appendix~\ref{app:proofs}.
The synthesis result enables a bisection search over $\rho$: for each
candidate $\rho$, feasibility of~\eqref{eq:IQC_convergenceSynth} is a convex
SDP.
With the optimal $\rho$ and corresponding $\lambda_0, \ldots, \lambda_s$,
the controller parameters $A_c, B_c, C_c, D_c$ are extracted by solving a
passivity design problem derived from~\eqref{eq:IQC_convergenceProj1}.

%===============================================================================
\section{Numerical Experiments}
\label{sec:numerics}
%===============================================================================

\subsection{Experimental setup}

Bisection determines the infimal rate \(\rho\) for which the synthesis conditions~\eqref{eq:IQC_convergenceSynth} are feasible.
For each candidate $\rho$, feasibility of the semidefinite
constraints~\eqref{eq:synthesisCondition_a}--\eqref{eq:synthesisCondition_d}
is checked using the MOSEK solver via CVXPY.
Bisection terminates when the interval width falls below $10^{-4}$. Throughout, we assume \(\tilde{b} = b\), which corresponds to sampling mini-batches with replacement, because this setting yields results that are independent of \(N\). For large \(N\), \(\tilde{b} \approx b\) is also a good approximation for sampling without replacement.

We use filter order $s = 1$, corresponding to a first-order FIR filter with $\lambda_0 = 1$ (fixed) and $\lambda_1 \leq 0$ (optimized). Numerical experiments with $s > 1$ did not improve convergence rates, suggesting that $s = 1$ is sufficient to capture the essential stochastic gradient structure.

We sweep condition numbers $\kappa = l/m$ from approximately $1.1$ to $10^4$, with $l = 1$ fixed and $m = 1/\kappa$ varying. For each $\kappa$, synthesis is performed for batch sizes $b \in \{1, 5, 25, 125\}$.

\subsection{Iteration complexity and evaluation complexity}

Algorithm speed is typically measured either per iteration \(t\) or per component gradient evaluation \(k = tb\); under Assumptions~\ref{ass:lmConvexity} and \ref{ass:irreducibleVariance}, optimal convergence is exponential in both cases.

In the per-iteration case, the fastest rate \(\rho\) such that
\begin{align*}
   \frac{\bbE\left[ \|z_t - z_\star^f\|^2 \right]}{\|z_0 - z_\star^f\|^2} \leq c \rho^{2t}
\end{align*}
holds for all \(z_0\), \(t\), \(N\) and a universal \(c\) is $\rho_\mathrm{TM} = 1 - \sqrt{m/l}$, attained by TM with full gradients ($b = N$) and proven optimal in~\cite{drori2022oracle}.

For the per-evaluation case, no optimal rate is known to the authors. To the best of the authors' knowledge (cf.~\cite{kidambi2018insufficiency}), the rate $\rho_\mathrm{GD} = (l-m)/(l+m)$ of GD with \(b = 1\) remains a candidate for the optimal \(\rho\) satisfying
\begin{align*}
   \frac{\bbE\left[ \|z_k - z_\star^f\|^2 \right]}{\|z_0 - z_\star^f\|^2} \leq c \rho^{2k}
\end{align*}
for fixed \(c\) and arbitrary \(z_0\), \(k\) and \(N\). Here, \(z_k\) is the query point after \(k\) component gradient evaluations.

The per-evaluation rate translates to the per-iteration rate \(\rho_\mathrm{GD}^b = ((l-m)/(l+m))^b\) via the identity \(k = tb\). If \(\rho_\mathrm{GD}\) is optimal per gradient evaluation, then $\rho_\mathrm{LB} = \max\{\rho_\mathrm{TM}, \rho_\mathrm{GD}^b\}$ is a combined lower bound on the per-iteration rate of any algorithm.
For $b = 1$ the GD bound is always more restrictive; for large $b$, the TM rate becomes dominant.

% Let \(t\) be the number of iterations and \(k = tb\) be the number of component gradient evaluations for the objective \eqref{eq:compositeObjective}. Convergence speed is typically understood as 
% \begin{align*}
%    \frac{\bbE\left[ \|z_t - z_\star^f\|^2 \right]}{\|z_0 - z_\star^f\|^2}.
% \end{align*}
% for a fixed number of iterations \(t\) or equivalently a fixed number of component gradient evaluations \(k\).

% Under Assumptions~\ref{ass:lmConvexity} and \ref{ass:stochasticGradients}, the Triple Momentum algorithm with full gradients 

% Two reference rates serve as benchmarks:
% \begin{itemize}
%    \item \textbf{Triple Momentum (TM):}
%          $\rho_\mathrm{TM} = 1 - \sqrt{m/l}$ is the fastest convergence rate
%          for gradient methods on $\scrS_{m,l}$ according
%          to~\cite{drori2022oracle}.
%    \item \textbf{Batch GD bound:}
%          $\rho_\mathrm{GD}^b = ((l-m)/(l+m))^b$ is motivated by~\cite{kidambi2018insufficiency}, which shows that momentum does not improve the SGD rate for $b = 1$. If this result generalizes to arbitrary algorithm structure, then GD is optimal for \(b = 1\). In this case, exceeding the rate of \(b\) sequential GD steps for \(b > 1\) would imply that the location of the optimum can be inferred more accurately from the mini-batch gradient than from \(b\) individual stochastic gradients, which is not possible.
% \end{itemize}

\begin{remark}
\label{rem:illConditionedRegimes}
For large $\kappa$ we obtain
\begin{align*}
   \rho_\mathrm{TM} = 1 - \kappa^{-1/2},
   \qquad
   \rho_\mathrm{GD}^b \approx 1 - 2b/\kappa.
\end{align*}
Hence, for fixed $b$ and large $\kappa$, the TM rate is smaller (better) than the lower bound on the per-iteration rate inferred from GD, which means that the TM bound may not be achievable for large $\kappa$ and small $b$.
\end{remark}

\subsection{Synthesis results}

Figure~\ref{fig:synthesisRates} shows the synthesized convergence rates $\rho$ plotted against the condition number $\kappa$ for different batch sizes. For $b = 1$, the synthesized rate equals $\rho_\mathrm{GD}$ for all $\kappa$.
As $b$ increases, the rates approach the deterministic TM rate, with the gap shrinking significantly for $b = 25$ and $b = 125$.

\begin{figure}[t]
   \centering
   \input{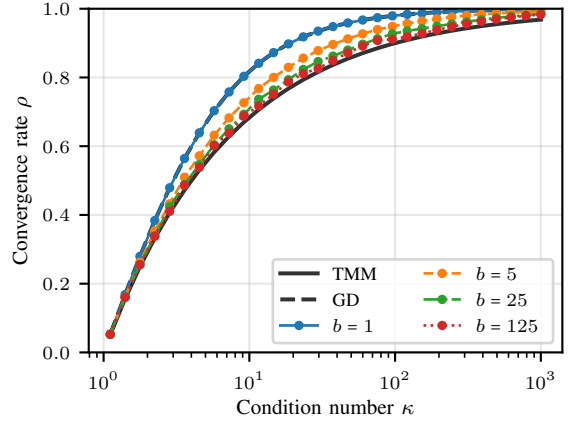}
   \caption{Synthesized convergence rates $\rho$ for stochastic gradient
            algorithms vs.\ condition number
            $\kappa = l/m$. The solid black line is the TM rate. The dashed black line, which is occluded by the $b = 1$ curve, is the GD rate. Curves for $b > 1$ approach the TM rate as $\kappa$ increases.}
   \label{fig:synthesisRates}
\end{figure}

Figure~\ref{fig:parallelismCost} shows the ratio
$\rho_\mathrm{synthesis}/\rho_\mathrm{GD}^b$ for $\kappa \geq 50$, quantifying the suboptimality of one parallel iteration relative to $b$ sequential gradient steps. This ratio always exceeds~$1$, consistent with the hypothetical lower bound, and approaches~$1$ for large $\kappa$. For example, at $\kappa = 10^4$, the ratios are approximately $1.000$ for $b = 5$, $1.003$ for $b = 25$, and $1.023$ for $b = 125$.

Together, Figures~\ref{fig:synthesisRates} and~\ref{fig:parallelismCost} reveal a tradeoff: for \(b = 1\), the synthesized algorithm reduces to gradient descent; as \(b\) grows, the per-evaluation rate degrades while the per-iteration rate approaches TM. This raises the question whether \(b > 1\) is beneficial, given that GD with \(b = 1\) has a superior per gradient evaluation rate? The canonical answer argues with parallelization: GD with \(b = 1\) requires sequential gradient evaluations, whereas \(b > 1\) admits \(b\) gradient evaluations in parallel. Viewed through this lens, Figure~\ref{fig:parallelismCost} makes a strong case for the synthesized algorithms: provided \(b \ll \kappa\), the synthesized algorithms enable \(b\) parallel gradient evaluations with a negligible loss in convergence speed.

The next subsection concerning gradient noise rejection provides another argument for algorithms with \(b > 1\).

\begin{figure}[t]
   \centering
   \input{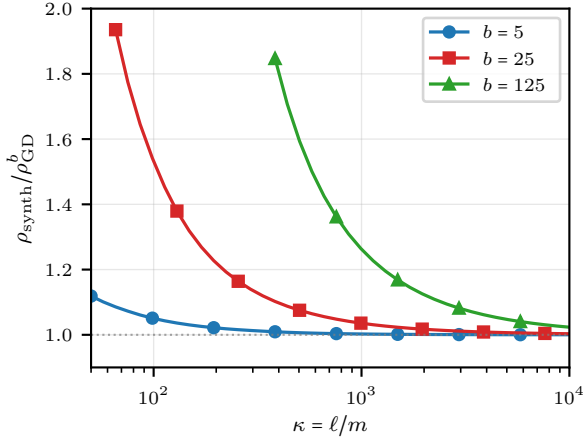}
   \caption{Cost of parallelism: ratio $\rho_\mathrm{synthesis}/\rho_\mathrm{GD}^b$
            vs.\ condition number $\kappa$ for $\kappa \geq 50$.
            A ratio of $1$ (horizontal reference line) corresponds to perfect
            parallelism efficiency.}
   \label{fig:parallelismCost}
\end{figure}

\subsection{Noise rejection analysis}

To illustrate Theorem~\ref{thm:asymptoticVarianceDynamic}, we compute asymptotic variance bounds for GD, NAGD, TM, (all with deterministic-optimal parameters) and the synthesized algorithm. The batch size is $b = 5$. For each $\kappa$, the theorem's SDP yields the tightest certified bound $\sigma_z^2$ on $\limsup \frac{1}{\tau}\sum_{t=0}^{\tau-1}\bbE\|z_t\|^2$ with $\sigma_\infty^2 = 1$.

Figure~\ref{fig:noiseRejection} shows that the variance bound for GD grows linearly with $\kappa$, remaining finite throughout, while the bounds for NAGD and TM grow rapidly and become infinite at a finite $\kappa$, indicating instability consistent with Fig.~\ref{fig:asymptoticVariance}.
The synthesized algorithm exhibits a smaller (better) variance bound than the other methods, despite not being explicitly optimized for noise rejection.
The slight irregularity of the synthesized curve arises because the solver may extract algorithms with varying noise rejection properties from the same rate certificate.

Note that the variance bound for GD with \(b = 1\) would be larger than the GD curve displayed in Figure~\ref{fig:noiseRejection}. Hence, synthesized methods with \(b > 1\) achieve not only a faster per-iteration convergence rate but also stronger noise rejection than GD with \(b = 1\).

\begin{figure}[t]
   \centering
   \input{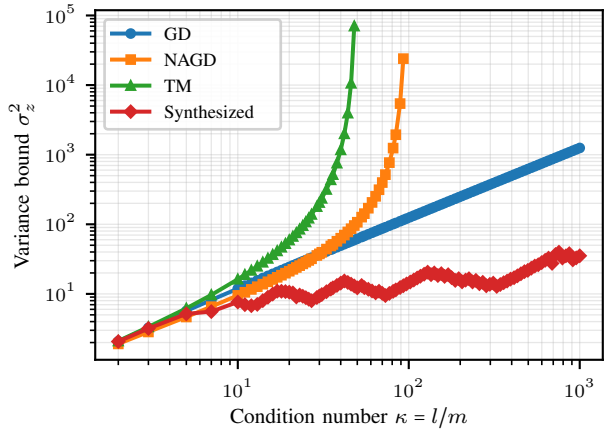}
   \caption{Variance bound $\sigma_z^2$ for GD, NAGD, TM, and the synthesized
            algorithm with batch size $b = 5$ and $\sigma_\infty^2 = 1$.
            Curves terminating before $\kappa_{\max}$ indicate instability.
            GD remains stable for all $\kappa$.}
   \label{fig:noiseRejection}
\end{figure}

%===============================================================================
\section{Conclusions}
\label{sec:conclusions}
%===============================================================================

This article develops a stochastic IQC framework for the analysis and synthesis of momentum-based gradient methods.
The key insight is that mini-batch gradient uncertainty can be captured by a generalization of Zames--Falb multipliers to the stochastic setting, and the resulting constraints admit a convex semidefinite programming formulation for both analysis and synthesis.

Numerically, the synthesized algorithms demonstrate two asymptotic behaviors: (i) as $b \to \infty$, the convergence rate approaches the optimal deterministic TM rate; (ii) as $\kappa \to \infty$, the algorithm efficiently utilizes parallel gradient information, approaching the rate of $b$ sequential stochastic GD steps from above.

Future directions include the joint synthesis of convergence rate and asymptotic variance, and the extension to non-strongly-convex objectives ($m = 0$).
%===============================================================================
\section*{Acknowledgments}
%===============================================================================

The authors thank Andrei Kharitenko and Naum Dimitrieski for valuable
technical discussions.

%===============================================================================
\appendix
%===============================================================================

\subsection{Auxiliary lemmas}
\label{app:auxiliaryLemmas}

The first auxiliary lemma relies on characterizing the inclusion of functions in \(\scrS_{m,l}\) through the next inequality. Specifically, \(f\) belongs to \(\scrS_{m,l}\) if and only if \(f^m = f(\cdot) - \frac{m}{2}\|\cdot\|^2\) and \(f^l = f(\cdot) - \frac{l}{2}\|\cdot\|^2\) satisfy for all \(z_1, z_2 \in \bbR^d\) that
\begin{align}
   \label{eq:slopeRestrictionCondition}
   (\nabla f^m(z_1) - \nabla f^m(z_2))^\top (\nabla f^l(z_1) - \nabla f^l(z_2)) \leq 0.
\end{align}
This inequality is the sum of (15) in \cite{scherer2023robust} switching \(u\) and \(y\).

% Using \eqref{eq:slopeRestrictionCondition}, we first prove the following preliminary result.

\begin{lemma}
   \label{lem:varianceBoundOne}
   Suppose Assumption~\ref{ass:lmConvexity} holds. Then, for any \(z \in \bbR^d\),
   \begin{align}
      -\nabla f^m(z)^\top \nabla f^l(z) \geq\, \frac{1}{N} \sum_{i=1}^N\, \| \nabla f_i(z) - \nabla f_i(z_\star^f)\|^2 - \| \nabla f(z)\|^2.
   \end{align}
\end{lemma}
\begin{proof}
   Pick any \(z \in \bbR^d\). The incremental sector condition \eqref{eq:slopeRestrictionCondition} for each \(f_i\) implies
   \begin{align*}
      0 &\leq -(\nabla f_i^m(z) - \nabla f_i^m(z_\star^f))^\top ( \nabla f_i^l(z) - \nabla f_i^l(z_\star^f) )\\
      &=
      % \begin{pmatrix}
      %    z - z_\star^f\\
      %    \nabla f_i(z) - \nabla f_i(z_\star^f)
      % \end{pmatrix}^\top
      \begin{pmatrix}
         \bullet
      \end{pmatrix}^\top
      \begin{pmatrix}
         -ml I_d & \frac{l+m}{2} I_d\\
         \frac{l+m}{2} I_d & -I_d
      \end{pmatrix}
      \begin{pmatrix}
         z - z_\star^f\\
         \nabla f_i(z) - \nabla f_i(z_\star^f)
      \end{pmatrix}.
   \end{align*}
   Next, add \(\| \nabla f_i(z) - \nabla f_i(z_\star^f)\|^2 - \| \nabla f(z)\|^2\) from both sides and obtain
   \begin{align*}
      % \begin{pmatrix}
      %    z - z_\star^f\\
      %    \nabla f_i(z) - \nabla f_i(z_\star^f)
      % \end{pmatrix}^\top
      \begin{pmatrix}
         \bullet
      \end{pmatrix}^\top
      \begin{pmatrix}
         -ml I_d & \frac{l+m}{2} I_d\\
         \frac{l+m}{2} I_d & 0
      \end{pmatrix}
      \begin{pmatrix}
         z - z_\star^f\\
         \nabla f_i(z) - \nabla f_i(z_\star^f)
      \end{pmatrix} - \| \nabla f(z)\|^2\\
      \geq \| \nabla f_i(z) - \nabla f_i(z_\star^f)\|^2 - \| \nabla f(z)\|^2.
   \end{align*}
   Finally, summing over \(i\) and dividing by \(N\) yields the desired result.
\end{proof}

\begin{lemma}
	\label{lem:varianceBoundBatchGradient}
	Suppose Assumption~\ref{ass:lmConvexity} holds. Let for all \(t = 0,\ldots,\tau\) the query point \(z_t\) be an arbitrary random variable on \(\bbR^d\) which is \(\calF_{t-1}\)-measurable. Then,
	\begin{align*}
		\bbE \left(\left\| w_t^1 \right\|^2 \middle| \calF_{t-1} \right) \leq -\frac{1}{\tilde{b}} (g_t^m)^\top g_t^l.
	\end{align*}
	holds, where \(\tilde{b}\) is the \emph{generalized batch size} given by
	\begin{itemize}
		\item \(b\) if the sets \(\scrI_t\) are sampled with replacement, and
		\item \(\frac{b(N-1)}{N-b}\) if the sets \(\scrI_t\) are sampled without replacement.
	\end{itemize}
\end{lemma}

\begin{proof}
   By \(\calF_{t-1}\) measurability of \(z_t\), the mini-batch \(\scrI_t\) is independent from \(z_t\). Hence, recalling that \(w_t^1\) corresponds to the sample average
   \begin{align*}
      w_t^1 =  \frac{1}{b}\sum_{i\in \scrI_t} (\nabla f_i(z_t) - \nabla f_i(z_\star^f) - \nabla f(z_t)),
   \end{align*}
   with \(\bbE[w_t^1 | \calF_{t-1}] = 0\), the conditional expectation of \(\|w_t^1\|^2\) is given by
   \begin{align*}
      \bbE\left[ \|w_t^1\|^2 \middle| \calF_{t-1} \right] &= \frac{1}{\tilde{b}} \frac{1}{N} \sum_{i=1}^N \left\| \nabla f_i(z_t) - \nabla f_i(z_\star^f) - \nabla f(z_t) \right\|^2\\
      &= \frac{1}{\tilde{b}} \frac{1}{N} \sum_{i=1}^N \left\| \nabla f_i(z_t) - \nabla f_i(z_\star^f) \right\|^2  - \|\nabla f(z_t)\|^2,
   \end{align*}
   where \(\tilde{b}\) is determined by the sampling method according to \cite{lohr2021sampling}.
   The right-hand side is upper bounded by \(-\frac{1}{\tilde{b}} (g_t^m)^\top g_t^l\) by Lemma~\ref{lem:varianceBoundOne}, which completes the proof.
\end{proof}

\begin{lemma}[Zames-Falb multipliers]
    \label{lem:doublyHyperdominantMultipliers}
    Let \(f \in \scrS_{m,l}^0\).
    Then, for any doubly hyperdominant matrix \(\widetilde{H} \in \bbR^{\tau + 1 \times \tau + 1}\), i.e., any matrix \(\widetilde{H}\) that satisfies
	\begin{align*}
		0 \dot{\leq} \widetilde{H} \bone, \quad 0 \dot{\leq} \widetilde{H}^\top \bone, \quad 0 \dot{\geq} \widetilde{H} - \diag \widetilde{H},
	\end{align*}
	the matrix \(\widetilde{\bH} = \widetilde{H} \otimes I_d\) satisfies the inequality
    \begin{align}
        \begin{pmatrix}
            \bz\\
            \bg
        \end{pmatrix}^\top
        \begin{pmatrix}
            -ml (\widetilde{\bH} + \widetilde{\bH}^\top) & (l+m) \widetilde{\bH}\\
            (l+m) \widetilde{\bH}^\top & - (\widetilde{\bH} + \widetilde{\bH}^\top)
        \end{pmatrix}
        \begin{pmatrix}
            \bz\\
            \bg
        \end{pmatrix} &\geq 0 \label{eq:zamesFalbInequality}
    \end{align}
	for all \(\bz = (z_0,z_1,\ldots,z_\tau)\) and \(\bg = (g_0,g_1,\ldots,g_\tau)\) with \(z_t \in \bbR^d\) and \(g_t = \nabla f(z_t)\).
\end{lemma}
\begin{proof}
	This result is extracted from \cite{fetzer2017absolute} building on \cite{willems1968some}. 
\end{proof}
Consult also the original (continuous-time) source \cite{zames1968stability} of the Zames-Falb multipliers or \cite{scherer2023optimization}, where they are presented in the context of optimization algorithms.

\begin{lemma}[\cite{HELMERSSON19993361}, Theorem 3]
	\label{lem:elimination}
	Consider the matrix inequality
	\begin{align}
		\begin{pmatrix}
			I_k\\
			U^\top \calK V + W
		\end{pmatrix}^\top \calP \begin{pmatrix}
		I_k\\
		U^\top \calK V + W
	\end{pmatrix} \prec 0
	\label{eq:QI}
	\end{align}
	and assume that \(\calP = \calP^\top\) is invertible with exactly \(k\) negative eigenvalues. Let \(U_{\perp}, V_{\perp}\) be basis matrices of \(\ker (U), \ker (V)\). Then there exists a \(\calK\) such that \eqref{eq:QI} is satisfied if and only if the matrix inequalities
	\begin{align*}
		V_{\perp}^\top \begin{pmatrix}
			I\\ W
		\end{pmatrix}^\top \calP \begin{pmatrix}
		I\\ W
	\end{pmatrix} V_{\perp} &\prec 0\\ 
	U_{\perp}^\top \begin{pmatrix}
		W^\top \\ -I
	\end{pmatrix}^\top \calP^{-1} \begin{pmatrix}
		W^\top\\ -I
	\end{pmatrix} U_{\perp} &\succ 0
	\end{align*}
	are satisfied.
\end{lemma}
\begin{lemma}
	\label{lem:transformationVector}
	Consider the filter matrices \eqref{eq:filterRealization}.
	Then the functions \(\bar{q}(\lambda)\), \(q(\lambda)\) from \eqref{eq:qFunctions} solve the equation system
	\begin{align*}
		\begin{pmatrix}
			(A_\psi^\top - \rho^{-1}I - C_\psi^\top D_\psi^{-\top} B_\psi^\top ) & \rho^{-1}C_\psi^\top D_\psi^{-\top}\\
			-B_\psi^\top & \rho^{-1}
		\end{pmatrix}
		\begin{pmatrix}
			\bar{q}(\lambda)\\
			q(\lambda)
		\end{pmatrix}
		=
		\begin{pmatrix}
			0\\
			\rho^{-1}\lambda_0
		\end{pmatrix}.
	\end{align*}
\end{lemma}
\begin{proof}
    The matrix $A_\psi^\top - \rho^{-1}I$ takes the form
	\begin{align*}
		\begin{pmatrix}
			-\rho^{-1} & 0 & \cdots & 0 & 0\\
			1 & -\rho^{-1} & \ddots & \vdots & \vdots\\
			0 & 1 & \ddots & 0 & 0\\
			\vdots & \ddots & \ddots & \rho^{-1} & 0\\
			0 & \cdots & 0 & 1 & -\rho^{-1}
		\end{pmatrix}
		=
		-
		\begin{pmatrix}
			\rho & 0 & \cdots & 0\\
			\rho^2 & \rho & \ddots & \vdots\\
			\vdots & \ddots & \ddots & 0\\
			\rho^s & \rho^{s-1} & \cdots & \rho
		\end{pmatrix}^{-1}
	\end{align*}
	Consequently, we have \((A_\psi^\top - \rho^{-1}I)\bar{q}(\lambda) =
		-\rho^{-1}C_\psi^\top\). Thus, we obtain the first equation by
	\begin{align*}
		(A_\psi^\top &- \rho^{-1}I - C_\psi^\top D_\psi^{-\top} B_\psi^\top )\bar{q}(\lambda)
		\\
		&= (A_\psi^\top - \rho^{-1}I)\bar{q}(\lambda) - C_\psi^\top D_\psi^{-\top} B_\psi^\top \bar{q}(\lambda)\\
		&= 
		-\rho^{-1}C_\psi^\top - C_\psi^\top D_\psi^{-\top} B_\psi^\top \bar{q}(\lambda)\\
		&=
		-\rho^{-1}C_\psi^\top - \rho^{-1}C_\psi^\top \frac{q(\lambda) - \lambda_0}{\lambda_0}\\
		&= -q(\lambda)\rho^{-1}C_\psi^\top D_\psi^{-\top}.
	\end{align*}
	Further, \(B_\psi^\top \bar{q}(\lambda)\) simply extracts the last element of \(\bar{q}(\lambda)\), which yields the second equation by
	\begin{align*}
		\rho^{-1}q(\lambda) - B_\psi^\top \bar{q}(\lambda)
		&=
		\rho^{-1}q(\lambda) - \rho^{-1} (\lambda_0 - q(\lambda))
		=
		\rho^{-1} \lambda_0.
	\end{align*}
	This concludes the proof. 
\end{proof}

\begin{lemma}
	\label{lem:ZamesFalbSchurDueToPostiveRealness}
	Consider the filter realization \eqref{eq:filterRealization} and let the parameters \(\lambda_0,\ldots,\lambda_s\) satisfy \(\lambda_1,\ldots,\lambda_s \leq 0\) and \(\sum_{j = 0}^s \rho^{-j}\lambda_j > 0\) for some \(\rho \in ]0,1]\). If \(\lambda_0 > 0\), then the matrix
	\begin{align*}
		A_\psi - B_\psi D_\psi^{-1} C_\psi = 
		\begin{pmatrix}
			0 & 1 & \cdots & 0\\
			\vdots & \ddots & \ddots & \vdots\\
			0 & \cdots & 0 & 1\\
			-\lambda_s/\lambda_0 & -\lambda_{s-1}/\lambda_0 & \cdots & -\lambda_1/\lambda_0
		\end{pmatrix}
	\end{align*}
	is Schur-stable.
\end{lemma}
\begin{proof}
	Conceptually, the lemma follows from the fact that the system defined by the filter matrices \(A_\psi, B_\psi, C_\psi, D_\psi\) is positive real.

	However, we provide the following elementary proof. The characteristic polynomial of the matrix \(A_\psi - B_\psi D_\psi^{-1} C_\psi\) is given by
	\begin{align*}
		\det \left( zI - A_\psi + B_\psi D_\psi^{-1} C_\psi \right) = z^s + \sum_{j = 1}^s \frac{\lambda_j}{\lambda_0} z^{s-j}.
	\end{align*}
	Since \(\lambda_1,\ldots,\lambda_s \leq 0\), the condition \(\sum_{j = 0}^s \rho^{-j}\lambda_j > 0\) implies
	\begin{align*}
		|\lambda_0| > \sum_{j = 1}^s |\lambda_j| \rho^{-j} \geq \sum_{j = 1}^s |\lambda_j|.
	\end{align*} 
	The last inequality follows from \(\rho \in ]0,1]\). Consequently, for any \(z \in \bbC\) with \(|z| \geq 1\), we have
	\begin{align*}
		|z|^s &> \sum_{j = 1}^s \frac{|\lambda_j|}{|\lambda_0|} |z|^{s-j} \geq \left|\sum_{j = 1}^s \frac{\lambda_j}{\lambda_0} z^{s-j}\right|.
	\end{align*}
	This implies that the characteristic polynomial does not have roots outside the open unit disk, which concludes the proof. 
\end{proof}

\begin{lemma}
	\label{lem:matringaleDifference}
	The stochastic processes \(\bw^1 = (w_0^1,w_1^1,\ldots,w_{\tau-1}^1)\) and \(\bw^2 = (w_0^2,w_1^2,\ldots,w_{\tau-1}^2)\), defined in \eqref{eq:gradientSignals}, satisfy for all \(t = 0,\ldots,\tau-1\)
	\begin{align*}
		\bbE \left[ w_t^1 \middle| \calF_{t-1} \right] &= 0, & \bbE \left[ w_t^2 \middle| \calF_{t-1} \right] &= 0.
	\end{align*}
	Consequently, \(\bw^1\) or \(\bw^2\) are not correlated with stochastic processes \(by = (y_0,y_1,\ldots,y_{\tau-1})\) which are adapted to \(\calF_{t-1}\), such as \(y_t = z_t\) or \(y_t = g_t\). Specifically, the correlation matrices vanish in expectation, i.e.,
   \begin{align*}
      \bbE \left[ w_t^1 y_t^\top \right] &= 0, & \bbE \left[ w_t^2 y_t^\top \right] &= 0.
   \end{align*}
\end{lemma}
\begin{proof}
	Since \(z_t\) is \(\calF_{t-1}\)-measurable and the mini-batch \(\scrI_t\) is independent of \(\calF_{t-1}\), the conditional expectations w.r.t. \(\calF_{t-1}\) are computed as follows.

   The conditional expectation of \(w_t^2 = \nabla f_{\scrI_t}(z_\star^f)\) is
	\begin{align}
		\bbE \left[ w_t^2 \middle| \calF_{t-1} \right] = \bbE \left[ \nabla f_{\scrI_t}(z_\star^f) \right] = \nabla f(z_\star^f) = 0, \label{eq:w2zeroCondExp}
	\end{align}
   because \(z_\star^f\) is not a random variable and the expectation of \(\nabla f_{\scrI_t}\) is \(\nabla f\).

	For \(w_t^1 = \nabla f_{\scrI_t}(z_t) - \nabla f(z_t) - \nabla f_{\scrI_t}(z_\star^f)\), the \(\calF_{t-1}\)-measurability of \(z_t\) yields
	\begin{align*}
		\bbE \left[ w_t^1 \middle| \calF_{t-1} \right]
		&= \bbE \left[ \nabla f_{\scrI_t}(z_t) \middle| \calF_{t-1} \right] - \nabla f(z_t) - \bbE \left[ \nabla f_{\scrI_t}(z_\star^f) \right]\\
		&= \nabla f(z_t) - \nabla f(z_t) = 0,
	\end{align*}
	where the second step uses the unbiasedness of the mini-batch gradient, i.e., \(\bbE[\nabla f_{\scrI_t}(z)] = \nabla f(z)\) for any fixed \(z\).

	For the product terms, let \(y_t\) be a stochastic process adapted to \(\calF_{t-1}\). By the tower property and \(\calF_{t-1}\)-measurability of \(y_t\),
	\begin{align*}
		\bbE \left[ w_t^k y_t^\top \right] = \bbE \left[ \bbE \left[ w_t^k \middle| \calF_{t-1} \right] y_t^\top \right] = 0
	\end{align*}
	for \(k = 1,2\). 
\end{proof}

\section{Proofs of main results}
\label{app:proofs}

\subsection{Proof of Theorem~\ref{thm:stochasticIQCFiniteHorizon}}

The multiplier inequality \eqref{eq:stochasticIQC} can be expressed as
\begin{align*}
	2\bbE \left[\big(\bg - m\bz\big)^\top \bH \big(l\bz - \bg\big)\right] -\tilde{b}\bbE \left[ (\bw^2)^\top \big(\diag(\mu) \otimes I_d\big) \bw^2 \right]\\
	+\sigma_\infty^2  \bone^\top \mu 
   - 2\tilde{b} \bbE \left[ (\bw^1)^\top \big(\diag(\nu) \otimes I_d\big) \bw^1 \right]
	\geq 0 .
\end{align*}

Using Definition~\ref{def:irreducibleVariance}, the sum of the second and third term can be upper bounded by zero. Further, dividing by 2, it remains to show that
\begin{align*}
	\bbE \left[\big(\bg - m\bz\big)^\top \bH \big(l\bz - \bg\big)\right] - \tilde{b}\bbE \left[(\bw^1)^\top \big(\diag(\nu) \otimes I_d\big) \bw^1\right]
\end{align*}
is larger than or equal to zero.
Since \(\nu \leq H \bone, \nu \leq H^\top \bone\), the matrix \(H - \diag(\nu)\) is already doubly hyperdominant. Consequently, \(\bbE \left[\big(\bg - m\bz\big)^\top \bH \big(l\bz - \bg\big)\right]\) can be written as
\begin{align*}
	\bbE \left[\big(\bg - m\bz\big)^\top (H - \diag(\nu)) \otimes I_d \big(l\bz - \bg\big)\right]\\
	\quad + \bbE \left[\big(\bg - m\bz\big)^\top \big(\diag(\nu) \otimes I_d\big) \big(l\bz - \bg\big)\right].
\end{align*}
Of these two terms, the first one is non-negative by Lemma~\ref{lem:doublyHyperdominantMultipliers}, because \(\bg\) is composed of the deterministic gradient \(\nabla f\) and \(H-\diag \nu\) is doubly hyperdominant. Consequently, the multiplier inequality holds if
\begin{align*}
	\bbE \left[\big(\bg - m\bz\big)^\top \big(\diag(\nu) \otimes I_d\big) \big(l\bz - \bg\big)\right]
	&\\
	- \tilde{b} \bbE \left[ (\bw^1)^\top \big(\diag(\nu) \otimes I_d\big) \bw^1 \right] &\geq 0.
\end{align*}
The fact that this last inequality holds is a direct consequence of Lemma~\ref{lem:varianceBoundBatchGradient}. Thus, the multiplier inequality holds. \qed

\subsection{Proof of Theorem~\ref{thm:ZamesFalb}}

Recalling \(g_t^l = 0\) for \(t < 0\) yields
\begin{align*}
	\sum_{t=0}^{\tau-1} \sum_{k = 0}^s \rho^{k-2t} \lambda_k (g_t^m)^\top g_{t-k}^l
	=
	(\bg^l)^\top \big(H \otimes I_d\big) \bg^m,
\end{align*}
where the matrix \(H\) is given by
\begin{align*}
	\begin{pmatrix}
		\lambda_0 & 0 & \cdots\\
		\lambda_1 \rho^{-1} & \lambda_0 \rho^{-1} & 0 & \cdots\\
		\lambda_2 \rho^{-2} & \lambda_1 \rho^{-2} & \lambda_0 \rho^{-2} & 0 & \cdots\\
		\vdots & & \ddots & \ddots & \ddots\\
		\lambda_s \rho^{-s} & \lambda_{s-1} \rho^{-s} & \cdots & \lambda_1 \rho^{-s} & \lambda_0 \rho^{-s}\\
		0 & \lambda_s \rho^{-s-1} & \lambda_{s-1} \rho^{-s-1} & \cdots & \lambda_1 \rho^{-s-1} & \ddots\\
		0 & 0 & \lambda_s \rho^{-s-2} & \cdots & \lambda_2 \rho^{-s-2} & \ddots\\
		\vdots & & & \ddots & \vdots & \ddots\\
		0 & 0 & 0 & \cdots & \lambda_s \rho^{-(\tau-1)} & \cdots & \lambda_0 \rho^{-(\tau-1)}
	\end{pmatrix}.
\end{align*}
Due to \(\sum_{j = 0}^s \rho^{-j}\lambda_j \geq 0\) and \(\lambda_1,\ldots,\lambda_s \leq 0\) the matrix \(H\) satisfies the conditions of Theorem~\ref{thm:stochasticIQCFiniteHorizon} for the specific choice of \(\nu \in \bbR^{\tau+1}\) with \(\nu_t = \sum_{j = 0}^s \rho^{-j}\lambda_j\) for all \(t = 0,\ldots,\tau\).

Thus, by Theorem~\ref{thm:stochasticIQCFiniteHorizon} with \(\mu = 0\), 
\begin{align*}
	\bbE (\bg^l)^\top \big(H \otimes I_d\big) \bg^m \geq \tilde{b} \sum_{t = 0}^\tau \nu_t \bbE \|w_t^2\|^2 = \tilde{b}\left(\sum_{j = 0}^s \rho^{-j}\lambda_j\right) \|\bw^2\|^2.
\end{align*}
This proves the claim. \qed

\subsection{Proof of Theorem~\ref{thm:synthesisConditions}}

Partition the certificate matrix \(P \in \bbR^{n_c + 1 + s \times n_c + 1 + s}\) and its inverse as
\begin{align}
	P = \begin{pmatrix}
		P_{11} & P_{12} & P_{13}\\
		P_{21} & P_{22} & P_{23}\\
		P_{31} & P_{32} & P_{33}
	\end{pmatrix} = \begin{pmatrix}
		\widetilde{P}_{11} & \widetilde{P}_{12} & \widetilde{P}_{13}\\
		\widetilde{P}_{21} & \widetilde{P}_{22} & \widetilde{P}_{23}\\
		\widetilde{P}_{31} & \widetilde{P}_{32} & \widetilde{P}_{33}
	\end{pmatrix}^{-1}
	\label{eq:certificatePartioning}
\end{align}
with blocks of dimensions \(n_c, 1, n_f\) and denote the blocks \(\begin{psmallmatrix}
	P_{22} & P_{23}\\
	P_{32} & P_{33}
\end{psmallmatrix}\) and \(\begin{psmallmatrix}
	\widetilde{P}_{22} & \widetilde{P}_{23}\\
	\widetilde{P}_{32} & \widetilde{P}_{33}
\end{psmallmatrix}\) by \(P_\psi\) and \(\widetilde{P}_\psi\), respectively. Then, by the definition of \(\calB_2\) in \eqref{eq:combinedAlgorithmAndFilter}, the matrix inequality \eqref{eq:IQC_convergenceProj2} is equivalent to to the last condition in \eqref{eq:synthesisCondition_d}. Since \(P_{22}\) needs to be positive definite, the conditions in \eqref{eq:synthesisCondition_d} further imply that the filter parameters satisfy 
\begin{align}
   \lambda_1,\ldots,\lambda_s &\leq 0, & \sum_{j = 0}^s \rho^{-j}\lambda_j &> 0. 
   \label{eq:strict_positiveRealness}
\end{align}
Equation~\ref{eq:strict_positiveRealness} shows that the conditions on the filter parameters \((\lambda_j)_{j=0}^s\) in Theorem~\ref{thm:convergenceRateDynamic} hold.

Hence, it remains to show that there exist algorithm parameters \(A_c, B_c, C_c, D_c\) such that \eqref{eq:IQC_convergenceProj1} is satisfied if and only if the conditions \eqref{eq:synthesisCondition_a}, \eqref{eq:synthesisCondition_b}, and \eqref{eq:synthesisCondition_c} are satisfied.

To transform \eqref{eq:IQC_convergenceProj1} into \eqref{eq:synthesisCondition_a}, \eqref{eq:synthesisCondition_b}, and \eqref{eq:synthesisCondition_c}, we first note that \eqref{eq:IQC_convergenceProj1} can be expressed as
\begin{align}
	\begin{pmatrix}
		I & 0\\
		0 & 1\\
		\calA &\calB_1\\
		\calC_1 & \calD_{11}
	\end{pmatrix}^\top
	\underbrace{
	\begin{pmatrix}
		-P & 0 & 0 & 0\\
		0 & 0 & 0 & 1\\
		0 & 0 & P & 0\\
		0 & 1 & 0 & 0
	\end{pmatrix}}_{\calP}
	\begin{pmatrix}
		I & 0\\
		0 & 1\\
		\calA &\calB_1\\
		\calC_1 & \calD_{11}
	\end{pmatrix} \prec 0
	\label{eq:factorizationForElimination}
\end{align}
The matrix \(\calP\) in the triple product \eqref{eq:factorizationForElimination} satisfies the inertia condition of Lemma~\ref{lem:elimination}. Next, we partition the system matrices in \eqref{eq:factorizationForElimination} as
\begin{align*}
	&
	\left(
	\begin{array}{c|c}
		\calA & \calB_1\\
		\hline
		\calC_1 & \calD_1
	\end{array}
	\right)
   =
	\left(
	\begin{array}{ccc|c}
		\rho^{-1}A_c & \rho^{-1}B_c & 0 & 0\\
		m \rho^{-1}C_c & \rho^{-1} + m\rho^{-1}D_c & 0 & \rho^{-1}\\
		(l-m)B_\psi C_c & (l-m)B_\psi D_c & A_\psi & -B_\psi\\
		\hline
		(l-m)D_\psi C_c & (l-m)D_\psi D_c & C_\psi & -D_\psi
	\end{array}
	\right) =\\
	&
	\underbrace{\begin{pmatrix}
		\rho^{-1}I & 0\\
		0 & m\rho^{-1}\\
		0 & (l-m)B_\psi\\
		\hline
		0 & (l-m)D_\psi
	\end{pmatrix}}_{U^\top}
	\underbrace{
	\begin{pmatrix}
		A_c & B_c\\
		C_c & D_c
	\end{pmatrix}}_{\calK}
	\underbrace{
	\begin{pmatrix}
		I & 0\\
		0 & 1\\
		0 & 0\\
		\hline
		0 & 0
	\end{pmatrix}^\top}_{V}
	+
	\underbrace{
	\left(
	\begin{array}{ccc|c}
		0 & 0 & 0 & 0\\
		0 & \rho^{-1} & 0 & \rho^{-1}\\
		0 & 0 & A_\psi & -B_\psi\\
		\hline
		0 & 0 & C_\psi & -D_\psi
	\end{array}
	\right)}_{W}.
\end{align*}
To apply Lemma~\ref{lem:elimination} we constructe the basis matrices $U_{\perp}, V_{\perp}$ of the kernels of $U, V$ given by
\begin{align*}
	U_{\perp} = \begin{pmatrix}
		0 & 0\\
		1 & 0\\
		0 & I\\
		\frac{-\rho^{-1}m}{l-m}D_\psi^{-\top} & -D_\psi^{-\top}B_\psi^\top 
	\end{pmatrix},
	\quad
	V_{\perp} = \begin{pmatrix}
		0 & 0\\
		0 & 0\\
		I & 0\\
		0 & 1
	\end{pmatrix}.
\end{align*}
Since the algorithm parameters \(A_c, B_c, C_c, D_c\) are arbitrary, this enables the elimination of \(\calK\) from \eqref{eq:factorizationForElimination} using Lemma~\ref{lem:elimination}. We obtain the reduced matrix inequalities
\begin{align*}
	% V_\perp^\top
	% \left(
	% \begin{array}{ccc|c}
	% 	I & 0 & 0 & 0\\
	% 	0 & 1 & 0 & 0\\
	% 	0 & 0 & I & 0\\
	% 	\hline
	% 	0 & 0 & 0 & I\\
	% 	\hline
	% 	0 & 0 & 0 & 0\\
	% 	0 & \rho^{-1} & 0 & \rho^{-1}\\
	% 	0 & 0 & A_\psi & -B_\psi\\
	% 	\hline
	% 	0 & 0 & C_\psi & -D_\psi
	% \end{array}
	% \right)^\top
	\begin{pmatrix}
		\bullet
	\end{pmatrix}^\top
	\begin{pmatrix}
		-P & 0 & 0 & 0\\
		0 & 0 & 0 & I\\
		0 & 0 & P & 0\\
		0 & I & 0 & 0
	\end{pmatrix}
	\left(
	\begin{array}{ccc|c}
		I & 0 & 0 & 0\\
		0 & 1 & 0 & 0\\
		0 & 0 & I & 0\\
		\hline
		0 & 0 & 0 & 1\\
		\hline
		0 & 0 & 0 & 0\\
		0 & \rho^{-1} & 0 & \rho^{-1}\\
		0 & 0 & A_\psi & -B_\psi\\
		\hline
		0 & 0 & C_\psi & -D_\psi
	\end{array}
	\right)
	V_\perp &\prec 0,
\end{align*}
\begin{align*}
	% U_\perp^\top
	% \left(
	% \begin{array}{ccc|c}
	% 	0 & \rho^{-1} & 0 & 0\\
	% 	0 & 0 & A_\psi^\top & C_\psi^\top\\
	% 	\hline
	% 	0 & \rho^{-1} & -B_\psi^\top & -D_\psi^\top\\
	% 	\hline
	% 	-I & 0 & 0 & 0\\
	% 	0 & -1 & 0 & 0\\
	% 	0 & 0 & -I & 0\\
	% 	\hline
	% 	0 & 0 & 0 & -1
	% \end{array}
	% \right)^\top
	\begin{pmatrix}
		\bullet
	\end{pmatrix}^\top
	\begin{pmatrix}
		-P & 0 & 0 & 0\\
		0 & 0 & 0 & 1\\
		0 & 0 & P & 0\\
		0 & 1 & 0 & 0
	\end{pmatrix}^{-1}
	\left(
	\begin{array}{ccc|c}
		0 & 0 & 0 & 0\\
		0 & \rho^{-1} & 0 & 0\\
		0 & 0 & A_\psi^\top & C_\psi^\top\\
		\hline
		0 & \rho^{-1} & -B_\psi^\top & -D_\psi^\top\\
		\hline
		-I & 0 & 0 & 0\\
		0 & -1 & 0 & 0\\
		0 & 0 & -I & 0\\
		\hline
		0 & 0 & 0 & -1
	\end{array}
	\right)
	U_\perp &\succ 0.
\end{align*}
Multiplying out the matrix products yields the conditions
\begin{align}
	\begin{pmatrix}
		I & 0\\
		\hline
		0 & 1\\
		\hline
		0 & \rho^{-1}\\
		A_\psi & -B_\psi\\
		\hline
		C_\psi & -D_\psi
	\end{pmatrix}^\top
	\begin{pmatrix}
		-P_{33} & 0 & 0 & 0\\
		0 & 0 & 0 & 1\\
		0 & 0 & P_\psi & 0\\
		0 & 1 & 0 & 0
	\end{pmatrix}
	\begin{pmatrix}
		I & 0\\
		\hline
		0 & 1\\
		\hline
		0 & \rho^{-1}\\
		A_\psi & -B_\psi\\
		\hline
		C_\psi & -D_\psi
	\end{pmatrix}&\prec 0,
	\label{eq:firstEliminationCondition}
\end{align}
\begin{align}
	% \begin{pmatrix}
	% 	(1-l/m) B_\psi^\top & (1-l/m)D_\psi^\top & -1/m\\
	% 	A_\psi^\top & C_\psi^\top & 0\\
	% 	-B_\psi^\top & -D_\psi^\top & 0\\
	% 	\rho(1-l/m)B_\psi^\top & \rho(1-l/m)D_\psi^\top & -\rho m^{-1}\\
	% 	I & 0 & 0\\
	% 	0 & 1 & 0\\
	% 	0 & 0 & 1
	% \end{pmatrix}
	\begin{pmatrix}
		\bullet
	\end{pmatrix}^\top
	\begin{pmatrix}
		-\widetilde{P}_\psi & 0 & 0 & 0\\
		0 & 0 & 0 & 1\\
		0 & 0 & \widetilde{P}_\psi & 0\\
		0 & 1 & 0 & 0
	\end{pmatrix}
	\begin{pmatrix}
		\rho^{-1} & 0\\
		\frac{-\rho^{-1} m}{l-m}C_\psi^\top D_\psi^{-\top} & A_\psi^\top - C_\psi^\top D_\psi^{-\top} B_\psi^\top\\
		\hline
		\frac{\rho^{-1}l}{l-m}D_\psi^{-\top} & 0\\
		\hline
		1 & 0\\
		0 & I\\
		\hline
		\frac{\rho^{-1}m}{l-m} & B_\psi^\top
	\end{pmatrix}
	&\succ 0.
	\label{eq:secondEliminationCondition}
\end{align}
The first of these inequalities \eqref{eq:firstEliminationCondition} is exactly \eqref{eq:synthesisCondition_a}.

The fact that neither of the inequalities \eqref{eq:firstEliminationCondition} or \eqref{eq:secondEliminationCondition} directly depends on the matrices \(P_{11}, P_{12}, P_{13}\) or \(\widetilde{P}_{11}, \widetilde{P}_{12}, \widetilde{P}_{13}\) enables the reparametrization of the certificate matrix \(P\) as
\begin{align*}
   P =
   \left(
   \begin{array}{c|cc}
      P_{11} & P_{12} & P_{13}\\
      \hline
      P_{21} & P_{22} & P_{23}\\
      P_{31} & P_{32} & P_{33}
   \end{array}
   \right)
   =
   \left(
   \begin{array}{c|c}
      P_\psi - \widetilde{P}_\psi^{-1} & P_\psi - \widetilde{P}_\psi^{-1}\\
      \hline
	   P_\psi - \widetilde{P}_\psi^{-1} & P_\psi
   \end{array}
   \right)
\end{align*}
in terms of the blocks \(P_\psi\) and \(\widetilde{P}_\psi\). Specifically, for any \(P_\psi\) and \(\widetilde{P}_\psi\) satisfying the condition
\begin{align}
	\begin{pmatrix}
		P_\psi & I\\
		I & \widetilde{P}_\psi
	\end{pmatrix}
	\succ 0,
	\label{eq:schurComplementPf}
\end{align}
the matrix \(P = \begin{psmallmatrix}
	P_\psi - \widetilde{P}_\psi^{-1} & P_\psi - \widetilde{P}_\psi^{-1}\\
	P_\psi - \widetilde{P}_\psi^{-1} & P_\psi
\end{psmallmatrix}\) is positive definite and its inverse has the right-lower block equal to \(\widetilde{P}_\psi\). The condition \eqref{eq:schurComplementPf} is also necessary for the existence of a positive definite matrix \(P\) whose right-lower block and inverse right-lower block are equal to \(P_\psi\) and \(\widetilde{P}_\psi\), respectively.

Clearly, this reparametrization assigns specific values to the blocks \(P_{11}, P_{12}, P_{13}\). However, since these blocks do not appear in the inequalities \eqref{eq:firstEliminationCondition} and \eqref{eq:secondEliminationCondition}, this does not affect the feasibility of these inequalities. Hence, we replace the search for \(P \succ 0\) satisfying \eqref{eq:IQC_convergenceProj1} by a search for \(P_\psi\) and \(\widetilde{P}_\psi\) satisfying \eqref{eq:schurComplementPf} along with the two inequalities \eqref{eq:firstEliminationCondition} and \eqref{eq:secondEliminationCondition}.

% The fact that neither of these inequalities directly depends on the matrices \(P_{11}, P_{12}, P_{13}\) or \(\widetilde{P}_{11}, \widetilde{P}_{12}, \widetilde{P}_{13}\) enables the elimination of these blocks. Specifically, for any \(P_\psi\) and \(\widetilde{P}_\psi\) satisfying the condition
% \begin{align}
% 	\begin{pmatrix}
% 		P_\psi & I\\
% 		I & \widetilde{P}_\psi
% 	\end{pmatrix}
% 	\succ 0,
% 	\label{eq:schurComplementPf}
% \end{align}
% the matrix \(P = \begin{psmallmatrix}
% 	P_\psi - \widetilde{P}_\psi^{-1} & P_\psi - \widetilde{P}_\psi^{-1}\\
% 	P_\psi - \widetilde{P}_\psi^{-1} & P_\psi
% \end{psmallmatrix}\) is positive definite and its inverse has the right lower block equal to \(\widetilde{P}_\psi\). The condition \eqref{eq:schurComplementPf} is also necessary for the existence of such a matrix \(P\). Hence, we can replace the search for \(P\) satisfying \eqref{eq:certificatePartioning} by a search for \(P_\psi\) and \(\widetilde{P}_\psi\) satisfying \eqref{eq:schurComplementPf} along with the two inequalities above.

It remains to show that \eqref{eq:secondEliminationCondition} and \eqref{eq:schurComplementPf} are equivalent to \eqref{eq:synthesisCondition_b} and \eqref{eq:synthesisCondition_c}. To this end, multiply \(\widetilde{P}_\psi\) from the right by 
\begin{align*}
   T 
   =
   \begin{pmatrix}
		q(\lambda) & 0\\
		-\frac{m}{l-m}\bar{q}(\lambda) & I
	\end{pmatrix}
\end{align*}
and from the right by \(T^\top\) to define
\begin{align*}
	\ovl{P}_f
	=
	\begin{pmatrix}
		\ovl{P}_{22} & \ovl{P}_{23}\\
		\ovl{P}_{32} & \ovl{P}_{33}
	\end{pmatrix}
	= \begin{pmatrix}
		q(\lambda) & 0\\
		-\frac{m}{l-m}\bar{q}(\lambda) & I
	\end{pmatrix}^\top
	\widetilde{P}_\psi
	\begin{pmatrix}
		q(\lambda) & 0\\
		-\frac{m}{l-m}\bar{q}(\lambda) & I
	\end{pmatrix}.
\end{align*}
The matrix inequalities \eqref{eq:schurComplementPf} and \eqref{eq:secondEliminationCondition} can be posed in terms of the transformed matrix \(\ovl{P}_f\) by multiplying \eqref{eq:schurComplementPf} from the left by \(\diag(I,T)^\top\) and from the right by \(\diag(I,T^\top)\) to obtain
\begin{align}
	\begin{pmatrix}
		P_{22} & P_{23} & q(\lambda) & 0\\
		P_{32} & P_{33} & -\frac{m}{l-m}\bar{q}(\lambda) & I\\
		q(\lambda) & -\frac{m}{l-m}\bar{q}(\lambda)^\top & \ovl{P}_{22} & \ovl{P}_{23}\\
		0 & I & \ovl{P}_{32} & \ovl{P}_{33}
	\end{pmatrix} \succ 0 \label{eq:transformedSchurComplementPf}
\end{align}
and by multiplying \eqref{eq:secondEliminationCondition} from the left by \(\diag(1,T)\) and from the right by \(\diag(1,T^\top)\) to obtain
\begin{align}
	\begin{pmatrix}
		\bullet
	\end{pmatrix}^\top
	\begin{pmatrix}
		-\ovl{P}_f & 0 & 0 & 0\\
		0 & 0 & 0 & 1\\
		0 & 0 & \ovl{P}_f & 0\\
		0 & 1 & 0 & 0
	\end{pmatrix}
	\begin{pmatrix}
		\rho^{-1} & 0\\
		* & A_\psi^\top - C_\psi^\top D_\psi^{-\top} B_\psi^\top\\
		\hline
		\frac{\rho^{-1}lq(\lambda)}{l-m}D_\psi^{-\top} & 0\\
		\hline
		1 & 0\\
		0 & I\\
		\hline
		\frac{\rho^{-1} m}{l-m}\lambda_0 & B_\psi^\top
	\end{pmatrix}
	&\succ 0. \label{eq:transformedSecondEliminationCondition}
\end{align}
Here \(*\) denotes \(\frac{-m}{l-m}((A_\psi^\top - C_\psi^\top D_\psi^{-\top} B_\psi^\top - \rho^{-1}I)\bar{q}(\lambda)+\rho^{-1}C_\psi^\top D_\psi^{-\top}q(\lambda))\), which evaluates to zero due to Lemma~\ref{lem:transformationVector}.

The final step to show the equivalence of \eqref{eq:schurComplementPf} and \eqref{eq:secondEliminationCondition} to \eqref{eq:synthesisCondition_b} and \eqref{eq:synthesisCondition_c} is the elimination of \(\ovl{P}_{33}\) as follows. First, choose \(\ovl{P}_{33}\) as solution of the Lyapunov equation in \(X\),
\begin{align*}
	(A_\psi -C_\psi^\top D_\psi^{-\top} B_\psi^\top)^\top X (A_\psi -C_\psi^\top D_\psi^{-\top} B_\psi^\top) - X = -\eta I,
\end{align*}
with parameter \(\eta > 0\) that can be chosen arbitrarily large.

Since the matrix \(A_\psi^\top -C_\psi^\top D_\psi^{-\top} B_\psi^\top\) is Schur by Lemma~\ref{lem:ZamesFalbSchurDueToPostiveRealness}, the solution to this Lyapunov equation exists and is of the form \(X_\eta = \eta X_\star\) for some positive definite matrix \(X_\star\). 

Under the choice of \(\ovl{P}_{33} = X_\eta\), \eqref{eq:transformedSecondEliminationCondition} evaluates to
\begin{align*}
   \begin{pmatrix}
      \ovl{P}_{22} - \rho^{-2} \ovl{P}_{22} + 2 \rho^{-2} \frac{ml q(\lambda)}{l-m}  & \bullet\\
      \ovl{P}_{32} - \rho^{-1} (A_\psi - B_\psi D_\psi^{-1} C_\psi)\ovl{P}_{32} + \frac{\rho^{-1} l q(\lambda)}{l-m}B_\psi D_\psi^{-\top} & \eta I
   \end{pmatrix},
\end{align*}
where the term \((\bullet)\) can be inferred from symmetry.

Clearly, we see that this matrix can be rendered positive definite by choosing \(\eta\) sufficiently large if and only if its upper-left block is positive definite, which is exactly \eqref{eq:synthesisCondition_b}. Moreover, since also \(\ovl{P}_{33}\) gets scaled by \(\eta\), for sufficiently large \(\eta\), \eqref{eq:transformedSchurComplementPf} is positive definite if and only if its upper-left \(3 \times 3\) block is positive definite, which is exactly \eqref{eq:synthesisCondition_c}.
\qed

\subsection{Lower bound calculation for mini-batch algorithms}
\label{app:detailsLowerBound}

In this subsection, we provide the details of the variance calculation for the specific composite objective~\eqref{eq:specificCompositeObjective} under GD, HB, NAGD and TM. This presents a lower bound on the worst-case variance of the iterates for these algorithms.

For the quadratic objective $f_i(z) = \tfrac{c_i}{2}\|z - z_i^f\|^2$, the
gradient of $f_i$ at the query point $z_t = C x_t$ is
\begin{align}
   \nabla f_i(C x_t) = c_i\,C x_t - c_i\,z_i^f.
   \label{eq:affineGradient}
\end{align}
The constant term $- c_i\,z_i^f = \nabla f_i(0)$ is the \emph{structural gradient noise}
at the optimum $z_\star^f = 0$.
Substituting~\eqref{eq:affineGradient} into~\eqref{eq:linearGradientAlgorithm}
yields the \emph{affine} closed-loop system for mini-batch $\scrI_t = \{i\}$
\begin{align}
   x_{t+1} = A_i\,x_t + b_i,
   \qquad
   A_i \coloneq A + B\,c_i\,C,
   \quad
   b_i \coloneq -B\,c_i\,z_i^f.
   \label{eq:affineClosedLoop}
\end{align}
To track $\bbE\|z_t\|^2 = C\,\bbE[x_t x_t^\top]\,C^\top$, we augment the
state with a constant component $x_t^0 \equiv 1$ via
$\tilde{x}_t \coloneq (1 ~ x_t^\top)^\top \in \bbR^{n+1}$.
The augmented system
\begin{align}
   \tilde{x}_{t+1} = \widetilde{A}_i\,\tilde{x}_t,
   \qquad
   \widetilde{A}_i \coloneq
   \begin{pmatrix} 1 & 0\\
   b_i & A_i \end{pmatrix},
   \label{eq:augmentedSystem}
\end{align}
is \emph{linear} in $\tilde{x}_t$, so the second moment
$\Sigma_t \coloneq \bbE[\tilde{x}_t\tilde{x}_t^\top]$ obeys the
\emph{homogeneous} stochastic equation
\begin{align}
   \Sigma_{t+1} = \sum_{i=1}^N p_i\,\widetilde{A}_i\,\Sigma_t\,\widetilde{A}_i^\top,
   \label{eq:augmentedSecondMoment}
\end{align}
with \emph{no} additive noise term.
The residual map $\Sigma\mapsto\sum_{i=1}^N p_i\widetilde{A}_i\Sigma\widetilde{A}_i^\top - \Sigma$ of \eqref{eq:augmentedSecondMoment} is
\emph{singular}: the unit eigenvalue of $\widetilde{A}_i$ (from the first row)
renders this Lyapunov equation operator non-invertible.
The unique meaningful fixed point of \eqref{eq:augmentedSecondMoment} has $e_1^\top\Sigma_\infty e_1 = 1$
(since $x_t^0 \equiv 1$) and is determined by the block decomposition
\begin{align}
   \Sigma_\infty
   = \begin{pmatrix} 1 & \mu_x^\top \\ \mu_x & \Sigma_{xx} \end{pmatrix},
   \label{eq:augmentedFixedPoint}
\end{align}
where $\mu_x = (I - \bar{A})^{-1}\bar{b}$ with
$\bar{A} = \sum_{i=1}^N p_i A_i$, $\bar{b} = \sum_{i=1}^N p_i b_i$
is the asymptotic mean, and $\Sigma_{xx}$ satisfies the \emph{standard}
(non-singular) stochastic Lyapunov equation
\begin{align}
   \Sigma_{xx} = \sum_{i=1}^N p_i\,A_i\,\Sigma_{xx}\,A_i^\top + Q_{\mathrm{eff}},
   \label{eq:regularLyapunov}
\end{align}
with effective noise covariance
$Q_{\mathrm{eff}} = \sum_{i=1}^N p_i\bigl(A_i\mu_x b_i^\top
  + b_i\mu_x^\top A_i^\top + b_i b_i^\top\bigr)$.
Figures~\ref{fig:nesterovConvergence} and~\ref{fig:asymptoticVariance} are
computed using the finite-horizon propagation~\eqref{eq:augmentedSecondMoment}
and the asymptotic fixed point~\eqref{eq:regularLyapunov}, respectively.

%===============================================================================
\bibliographystyle{IEEEtran}
\bibliography{references}
%===============================================================================

\end{document}